\documentclass[12pt]{amsart}
\usepackage{amsmath,amsthm,amssymb,amsfonts,xcolor}
\usepackage{amssymb,epsfig,amsfonts}

\usepackage{amsmath,amscd,amsthm,array,geometry}

\title{Twisted limit formula for torsion and cyclic base change}

\author{Nicolas Bergeron} 

\address{Universit\'e Pierre et Marie Curie \\
Institut de Math\'ematiques de Jussieu-Paris Rive Gauche \\
4, place Jussieu 75252 Paris Cedex 05, France \\}
\email{bergeron@math.jussieu.fr}
\urladdr{http://people.math.jussieu.fr/~bergeron}

\author{Michael Lipnowski}
\thanks{N.B. is a member of the Institut Universitaire de France}
\address{}
\address{Mathematics Department \\
Duke University, Box 90320 \\
Durham, NC 27708-0320, USA \\}
\email{malipnow@math.duke.edu}

\newcommand{\loc}{\mathcal{L}}
\newcommand{\M}{\mathcal{M}}
\newcommand{\Hom}{\mathrm{Hom}}
\newcommand{\ind}{\mathrm{ind}}
\newcommand{\g}{\mathfrak{g}}
\newcommand{\h}{\mathfrak{h}}
\renewcommand{\a}{\mathfrak{a}}
\renewcommand{\t}{\mathfrak{t}}

\newcommand{\RR}{\mathbb{R}}
\newcommand{\p}{\mathfrak{p}}

\newcommand{\C}{\mathbb{C}}

 \DeclareFontFamily{OT1}{rsfs}{}
 \newcommand{\tors}{\mathrm{tors}}
\DeclareFontShape{OT1}{rsfs}{n}{it}{<-> rsfs10}{}
\DeclareMathAlphabet{\mathscr}{OT1}{rsfs}{n}{it}

\newcommand{\detp}{\det{}'}

\renewcommand{\a}{\mathfrak{a}}

\newcommand{\Gal}{\mathrm{Gal}}

\newcommand{\n}{\mathfrak{n}}

\DeclareFontFamily{OT1}{rsfs}{}

\DeclareFontShape{OT1}{rsfs}{n}{it}{<-> rsfs10}{}
\DeclareMathAlphabet{\mathscr}{OT1}{rsfs}{n}{it}

\newcommand{\G}{\mathbf{G}}
\renewcommand{\k}{\mathfrak{k}}

\newcommand{\R}{\mathbb{R}}
\newcommand{\SO}{\mathrm{SO}}

\newcommand{\Aut}{\mathrm{Aut}}

\swapnumbers
\newtheorem{thm}[subsection]{Theorem}  
\newtheorem{lem}[subsection]{Lemma}         
\newtheorem*{lem*}{Lemma}         
\newtheorem{prop}[subsection]{Proposition}
\newtheorem*{prop*}{Proposition}

\newtheorem{exa}[subsection]{Example}

\newtheorem{cor}[subsection]{Corollary}

\theoremstyle{definition}
\newtheorem{defn}[subsection]{Definition}

\numberwithin{equation}{subsection}

\newcommand{\N}{\mathbb N}

\newcommand{\GL}{\mathrm{GL}}
\newcommand{\SL}{\mathrm{SL}}

\newcommand{\tr}{\mathrm{trace}}

\newcommand{\vol}{\mathrm{vol}}

\begin{document}

\begin{abstract}  
Let $G$ be the group of complex points of a real semi-simple Lie group whose fundamental rank is equal to $1$, e.g. $G= \SL_2 ( \C) \times \SL_2 (\C)$ or $\SL_3 (\C)$. Then 
the fundamental rank of $G$ is $2,$ and according to the conjecture made in \cite{BV}, lattices in $G$ should have `little' --- in the very weak sense of `subexponential in the co-volume' --- torsion homology. Using base change, we exhibit sequences of lattices where the torsion homology grows exponentially with the \emph{square root} of the volume. 

This is deduced from a general theorem that compares twisted and untwisted $L^2$-torsions in the general base-change situation. This also makes uses of a precise equivariant `Cheeger-M\"uller Theorem' proved by the second author \cite{Lip1}.
\end{abstract}
\maketitle
\tableofcontents

\section{Introduction}

\subsection{Asymptotic growth of cohomology}

Let $\Gamma$ be a torsion-free uniform lattice in a semisimple Lie group $G$ with maximal compact subgroup $K.$  Let $\Gamma_n \subset \Gamma$ be a decreasing sequence of normal subgroups with trivial intersection.  It is known that
$$\lim_{n \rightarrow \infty} \frac{\dim H^j(\Gamma_n, \mathbb{C})}{[\Gamma : \Gamma_n]}$$
converges to $b^{(2)}_j(\Gamma),$ the $j$th $L^2$-betti number of $\Gamma.$  If $b^{(2)}_j \neq 0$ for some $j,$ it follows that cohomology is abundant.  However, it is often true that $b^{(2)}_j(\Gamma) = 0$ for all $j;$ this is the case whenever $\delta(G) := \mathrm{rank}_{\C} G - \mathrm{rank}_{\C}K \neq 0.$   What is the true rate of growth of $b_j(\Gamma_n) = \dim H^j(\Gamma_n, \C)$ when $\delta(G) \neq 0$?  In particular, is $b_j(\Gamma_n)$ non-zero for sufficiently large $n$?

We address this question for `cyclic base-change.'  Before stating a general result, let's give two typical examples of this situation.

\medskip
\noindent
{\it Examples.} 1. The real semisimple Lie group $G=\SL_2 (\C)$ satisfies $\delta = 1$. Let $\sigma : G \to G$ be the real involution given by complex conjugation. 
 
2. The real semisimple Lie group $G=\SL_2 (\C) \times \ldots \times \SL_2 (\C)$ ($n$ times) satisfies $\delta = n$. Let $\sigma : G \to G$ be the order $2n$ automorphism of $G$ given by $\sigma (g_1 , \ldots , g_n ) = (\bar{g}_n , g_1 , \ldots , g_{n-1})$.

\medskip

Now let $\Gamma_n \subset \Gamma$ be a sequence of finite index, $\sigma$-stable subgroups of $G$. It follows from the general Proposition \ref{prop:intro} below that 
$$\sum_j b_j (\Gamma_n ) \gg \vol (\Gamma_n^{\sigma} \backslash \SL_2 (\RR) ).$$
Note that when the $\Gamma_n$'s are congruence subgroups of an arithmetic lattice $\Gamma$, then $\vol (\Gamma_n^{\sigma} \backslash \SL_2 (\RR) )$ grows like
$\vol (\Gamma_n \backslash G )^{\frac{1}{\mathrm{order}(\sigma)}}$. 

In this paper, we shall more generally consider the case where $G$ is obtained from a real algebraic group by `base change.' Let $\mathbf{G}$ be a connected 
semisimple quasi-split algebraic group defined over $\RR$. Let $\mathbb{E}$ be an \'{e}tale $\R$-algebra such that $\mathbb{E} / \RR$ is a cyclic Galois
extension with Galois group generated by $\sigma \in \Aut(\mathbb{E} / \R).$ Concretely, $\mathbb{E}$ is either $\mathbb{R}^s$ or $\mathbb{C}^s.$ 
In the first case $\sigma$ is of order $s$ and acts on $\mathbb{R}^s$ by cyclic permutation. In the second case $\sigma$ is of order $2s$ and acts on 
$\C^s$ by $(z_1 , \ldots , z_s) \mapsto (\bar{z}_s , z_1 , \ldots , z_{s-1})$.
The automorphism $\sigma$ induces a corresponding automorphism of the group $G$ of real points of $\mathrm{Res}_{\mathbb{E}/ \RR} \mathbf{G}$.\footnote{The assumption that $\G / \RR$  is quasi-split is used in the case where $\mathbb{E} = \mathbb{C},$ where we quote results of Delorme \cite{Delorme} concerning base change from $\G(\RR)$ to $\G(\mathbb{C}).$  We emphasize that assuming $\G / \RR$ is quasi-split is unnecessary in the case where $\mathbb{E} = \RR^s$ and $\sigma$ acts by cyclic permutation.  See \S \ref{product}.} We will furthermore assume that $H^1 (\sigma , G) = \{1 \}$; see \S \ref{par:C} for comments
on this condition. The following proposition is `folklore' (see e.g. Borel-Labesse-Schwermer \cite{BLS}, Rohlfs-Speh \cite{RohlfsSpeh} and Delorme \cite{Delorme}). 

\begin{prop} \label{prop:intro}
Let $\Gamma_n \subset \Gamma$ be a sequence of finite index, $\sigma$-stable subgroups of $G.$  Suppose that $\delta(G^{\sigma} ) = 0$. 
Then we have:
$$\sum_j \dim b_j (\Gamma_n ) \gg \vol (\Gamma_n^{\sigma} \backslash G^{\sigma}).$$
\end{prop}

We prove Proposition \ref{prop:intro} for certain families $\{ \Gamma_n \}$ in \S \ref{WLF} but our real interest here is rather how the {\it torsion cohomology} grows.

\subsection{Asymptotic growth of torsion cohomology}
Let $(\rho, F)$ be a finite dimensional representation of $G$ defined over $\mathbb{R}$ and suppose that the $\Gamma_n$'s stabilize some fixed lattice $\mathcal{O} \subset F.$  The first named author and Venkatesh \cite{BV} prove that for `strongly acyclic' \cite[$\S 4$]{BV} representations $\rho,$ there is a lower bound
$$\sum_j \log | H^j(\Gamma_n, \mathcal{O})_{\mathrm{tors}}| \gg c(G,\rho) \cdot [\Gamma : \Gamma_n].$$
for some constant $c(G,\rho).$  In fact, they prove a limiting identity
\begin{equation} \label{torsioneulerchar}
\frac{\sum_j (-1)^j \log | H^j(\Gamma_n, \mathcal{O})_{\mathrm{tors}}|}{[\Gamma:\Gamma_n]} \rightarrow c(G,\rho)
\end{equation}
and prove that $c(G,\rho)$ is non-zero exactly when $\delta(G) = 1.$  The numerator of the left side of \eqref{torsioneulerchar} should be thought of as a `torsion Euler characteristic.'  
The purpose of this article is to prove an analogous theorem about `torsion Lefschetz numbers.'  

To state one instance of our main result, we will assume that:
\begin{enumerate}
\item $\sigma$ has prime order $p$ and $\mathcal{O}_{\mathbb{F}_p}$ is trivial.
\item the $\Gamma_n$'s are $\sigma$-stable principal congruence of $\Gamma$ of level $q_n$ for some infinite sequence of primes $q_n$ (see \eqref{congp}), and
\item the representation $\rho$ is strongly acyclic and can be extended to a finite dimensional (twisted) representation $\tilde{\rho}$ of the twisted space $\widetilde{G} = G \rtimes \sigma$ that is {\it strongly acyclic} (see \S \ref{def:SA}).
\end{enumerate}

Under these hypotheses we shall prove the following:

\begin{thm} \label{T:14}
We have: 
\begin{equation} \label{coarsetorsion}
\limsup \frac{\sum_j \log |H^j(\Gamma_n, \mathcal{O})|}{\vol(\Gamma_n^{\sigma} \backslash G^{\sigma})} > 0
\end{equation}
whenever $\delta(G^{\sigma}) = 1$.
\end{thm}

\medskip
\noindent
{\it Example.} Let $E/\mathbb{Q}$ be an imaginary quadratic extension. Let $B$ be a division algebra of dimension 9 over $\mathbb{Q}$ such that $B$ is split at infinity and $B_E := B \otimes_{\mathbb{Q}} E$ is a division algebra. Let $\mathfrak{o}$ be a maximal order in $B$ and let $\mathfrak{o}_E$ be its tensor product over $\mathbb{Z}$ with the ring of integers of $E$. Then $\mathfrak{o}_E^{\times}$ embeds into $\mathrm{PGL}_3 (\mathbb{C})$. Let $\mathcal{O}$ be the set of elements in $\mathfrak{o}_E$ of trace $0$, considered as a $\mathfrak{o}_E^{\times}$-module by conjugation. Let $\Gamma$ be a torsion free congruence subgroup of $\mathfrak{o}_E^{\times}$ such that $\Gamma$ is contained in the kernel of $\rho$ mod $2$. Then the local system $\mathcal{O}_{\mathbb{F}_2}$ is trivial. Given a prime $q,$ we denote by $\Gamma_{q}$ the kernel of the reduction map $\Gamma \mapsto \left( \mathfrak{o}_E / q \mathfrak{o}_E \right)^{\times} $. Theorem \ref{T:14} applies to the sequence $\{ \Gamma_{q} \}$ and we conclude that
\begin{equation} \label{growthEx}
\limsup \frac{1}{q^8} \sum_j \log |H^j(\Gamma_q, \mathcal{O}) | > 0 .
\end{equation}
Here $q^8$ is the growth rate of the log of torsion in $H^3$ of the corresponding $q$-congruence subgroups of $\mathfrak{o}^{\times}$ embedded into $\mathrm{PGL}_3 (\mathbb{R})$. The cohomology classes that contribute to \eqref{growthEx} should conjecturally arise by base change transfer over $\mathbb{Z}$. One may regard \eqref{growthEx} as a partial evidence for the existence of such a transfer.

\medskip

Similarly, one can construct examples of lattices $\Gamma$ in $\mathrm{SL}_2 (\mathbb{C})^p$ ($p>1$ prime), in $\mathrm{SL}_3 (\mathbb{C})$ or in $\mathrm{SL}_4 (\mathbb{C})$ such that the torsion homology of level $q$ congruence subgroups of $\Gamma$ grows exponentially with respectively $q^6$, $q^8$ or $q^{15}$. 

\subsection{} Analogously to \eqref{torsioneulerchar}, we prove \eqref{coarsetorsion} by proving a limiting identity for torsion Lefschetz numbers.  For example, when $\sigma^2 = 1,$ conditional on an assumption about the growth of the Betti numbers $\dim_{\mathbb{F}_2} H^j(\Gamma_n^{\sigma}, \mathcal{O}_{\mathbb{F}_2})$ we prove that
\begin{equation} \label{torsionlef}
\frac{\sum (-1)^j (\log|H^j(\Gamma_n, \mathcal{O})^{+}_{\tors}| - \log |H^j(\Gamma_n, \mathcal{O})^{-}_{\mathrm{tors}}|) }{|H^1 (\sigma , \Gamma_n)| \vol (\Gamma_n^{\sigma} \backslash G^{\sigma})} \rightarrow c(G,\rho,\sigma),
\end{equation}
where the superscript $\pm$ denotes the $\pm 1$ eigenspace.  

Assume that the maximal compact subgroup $K \subset G$ is $\sigma$-stable and let $X = G/K$ and $X^{\sigma} =G^{\sigma} / K^{\sigma}$. The proof of \eqref{torsionlef} crucially uses the equivariant Cheeger-M\"{u}ller theorem, proven by Bismut-Zhang \cite{BZ2}.  This allows us to compute the left side of \eqref{torsionlef} (up to a controlled integer multiple of $\log 2$) by studying the eigenspaces of the Laplace operators of the metrized local system associated to $\rho$ together with their $\sigma$ action.  More precisely, the left side of \eqref{torsionlef} nearly equals the equivariant analytic torsion $\log T^{\sigma}_{\Gamma_n \backslash X} (\rho)$; see \eqref{TT} for a definition of the latter.  Using the simple twisted trace formula and results of Bouaziz \cite{Bouaziz}, we prove a `limit multiplicity formula.'

\begin{thm}
Assume (1) and (2) above. Then we have:
\begin{equation} \label{limitmult}
\frac{\log T^{\sigma}_{\Gamma_n \backslash X}(\rho)}{|H^1 (\sigma , \Gamma_n)| \vol (\Gamma_n^{\sigma} \backslash G^{\sigma})} \rightarrow  s 2^{r} t^{(2)}_{X^{\sigma}} (\rho),
\end{equation}
where $\mathbb{E} = \RR^s$ or $\C^s$, and $r = 0$ in the first case and $r=\mathrm{rank}_{\RR} \G (\C) - \mathrm{rank}_{\RR} \G (\R)$ in the second case. 
\end{thm}
Here $t^{(2)}_{X^{\sigma}} (\rho)$ is the (usual) $L^2$-analytic torsion of the symmetric space $X^{\sigma}$ twisted by the finite dimensional representation $\rho$. It is explicitly computed in \cite{BV}. Note that it is non-zero if and only if $\delta (G^{\sigma}) =1$.

The authors hope that the limit multiplicity formula \eqref{limitmult} together with the twisted endoscopic comparison implicit in Section 7 will be of interest independent of torsion in cohomology.  These computations complement work by Borel-Labesse-Schwermer \cite{BLS} and Rohlfs-Speh \cite{RohlfsSpeh}.

\medskip

{\it The authors would like to thank Abederrazak Bouaziz, Laurent Clozel, Colette Moeglin and David Renard for helpful conversations. They also would to thank an anonymous referee for pointing out an error in the first version of this paper.}

\medskip

\section{The simple twisted trace formula}
Let $\mathbf{G}$ be a connected 
semisimple quasi-split algebraic group defined over $\RR$. Let $\mathbb{E}$ be an \'{e}tale $\R$-algebra such that $\mathbb{E} / \RR$ is a cyclic Galois
extension with Galois group generated by $\sigma \in \Aut(\mathbb{E} / \R).$ The automorphism $\sigma$ induces a corresponding automorphism of the group $G$ of real points of $\mathrm{Res}_{\mathbb{E}/ \RR} \mathbf{G}$. We furthermore choose a Cartan involution $\theta$ of $G$
that commutes with $\sigma$ and denote by $K$ the group of fixed points of $\theta$ in $G$. Here we follow Labesse-Waldspurger \cite{LabesseWaldspurger}.

\subsection{Twisted spaces} We associate to these data the {\it twisted space} 
$$\widetilde{G} = G \rtimes \sigma \subset G \rtimes \mathrm{Aut} (G).$$
The left action of $G$ on $G$,
$$(g , x \rtimes \sigma ) \mapsto gx \rtimes \sigma,$$
turns $\widetilde{G}$ into a left principal homogeneous $G$-space equipped with a 
$G$-equivariant map $\mathrm{Ad} : \widetilde{G} \to \mathrm{Aut} (G)$ given by 
$$\mathrm{Ad} (x \rtimes \sigma) (g) = \mathrm{Ad} (x) (\sigma (g)).$$ 
We also have a right action of $G$ on $\widetilde{G}$ by 
$$\delta g = \mathrm{Ad} (\delta ) (g) \delta \quad (\delta \in \widetilde{G} , \ g \in G ).$$
This allows to define an action by conjugation of $G$ on $\widetilde{G}$ and yields a notion of $G$-conjugacy class in $\widetilde{G}$. 
Note that taking $\delta = 1 \rtimes \sigma$ we have: 
$$(1 \rtimes \sigma) g = \delta g = \sigma (g) \delta = \sigma (g) \rtimes \sigma.$$ 

We similarly define the twisted space $\widetilde{K} = K \rtimes \sigma$. 

\subsection{Twisted representations} 
A {\it representation of $\widetilde{G}$}, in a vector space $V$, is the data
for every $\delta \in \widetilde{G}$ of a invertible linear map 
$$\widetilde{\pi} (\delta) \in \mathrm{GL} (V)$$
and of a representation of $G$ in $V$:
$$\pi : G \to \mathrm{GL} (V)$$
such that for $x,y \in G$ and $\delta \in \widetilde{G}$,
$$\widetilde{\pi} (x \delta y) = \pi (x) \widetilde{\pi} (\delta ) \pi (y).$$
In particular 
$$\widetilde{\pi} (\delta x) = \pi ( \mathrm{Ad} (\delta ) (x)) \widetilde{\pi} (\delta).$$
Therefore $\widetilde{\pi} (\delta )$ intertwines $\pi$ and $\pi \circ  \mathrm{Ad} (\delta )$. Note that $\widetilde{\pi}$ determines $\pi$; we will
say that $\pi$ is the restriction of $\widetilde{\pi}$ to $G$. 

Conversely $\widetilde{\pi}$ is determined by the data of $\pi$ and of an operator $A$ which intertwines
$\pi$ and $\pi \circ \sigma$:
$$A \pi (x)  =(\pi \circ \sigma ) (x) A$$
and whose $p$-th power is the identity, where $p$ is the order of $\sigma$.
We reconstruct $\widetilde{\pi}$ by setting 
$$\widetilde{\pi} (x \rtimes \sigma ) = \pi (x) A \quad \mbox{for } x \in G.$$

Say that $\widetilde{\pi}$ is \emph{essential} if $\pi$ is irreducible. 
If $\widetilde{\pi}$ is unitary and essential, Schur's lemma implies that $\pi$ determines $A$
up to a $p$-th root of unity. 

There is a natural notion of equivalence between representations of $\widetilde{G}$ --- see e.g. \cite[\S 2.3]{LabesseWaldspurger}. This is the obvious one;
beware however that even if $\widetilde{\pi}$ is essential the class of $\pi$ does not determine the class of $\widetilde{\pi}$ since the intertwiner 
$A$ is only determined up to a root of unity.

We have a corresponding notion of a $(\mathfrak{g} , \widetilde{K})$-module.

If $\widetilde{\pi}$ is unitary and $f \in C_c^{\infty} (\widetilde{G})$ we set 
$$\widetilde{\pi} (f) = \int_{\widetilde{G}} f( y) \widetilde{\pi} (y) dy := \int_{G} f (x\rtimes \sigma)  \widetilde{\pi} (x \rtimes \sigma) dx.$$
It follows from \cite[Lemma 2.3.2]{LabesseWaldspurger} that $\widetilde{\pi} (f)$ is of trace class. Moreover:
$\mathrm{trace} \ \widetilde{\pi} (f) =0$ unless $\widetilde{\pi}$ is essential. 
In the following, we denote by $\Pi (\widetilde{G})$ the set of irreducible unitary representations $\pi$ of $G$ (considered up to equivalence) that
can be extended to some (twisted) representation $\widetilde{\pi}$ of $\widetilde{G}$. Note that the extension is not unique. 

\subsection{Twisted trace formula (in the cocompact case)} \label{TTF1}
Let $\Gamma$ be a cocompact lattice of $G$ that is $\sigma$-stable. Associated to $\Gamma$ is the (right) regular representation
$\widetilde{R}_\Gamma$ of $\widetilde{G}$ on $L^2(\Gamma \backslash G )$ where the restriction $R_\Gamma$ of $\widetilde{R}_\Gamma$ is the usual 
regular representation in $L^2(\Gamma \backslash G)$ and 
$$(\widetilde{R}_\Gamma (\sigma )) (f) (\Gamma x) = f( \Gamma \sigma (x)).$$
Note that
$$(\widetilde{R}_\Gamma (\sigma ) R_\Gamma (g )) (f) (\Gamma x) = f( \Gamma \sigma (x) g) = (R_\Gamma  (\sigma (g)) \widetilde{R}_\Gamma (\sigma ))(f) (\Gamma x).$$

Given $\delta \in \widetilde{G}$ we denote by $G^\delta$ its centralizer in $G$ (for the (twisted) 
action by conjugation of $G$ on $\widetilde{G}$). Corresponding to $\Gamma$ is a (non-empty) discrete twisted subspace 
$\widetilde{\Gamma} \subset \widetilde{G}$. Given $\delta \in \widetilde{\Gamma}$ we denote by $\{ \delta \}$ its $\Gamma$-conjugacy class (where here again
$\Gamma$ acts by (twisted) conjugation on $\widetilde{\Gamma}$). 

Let $f \in C_c^{\infty} (\widetilde{G})$. The twisted trace formula is obtained by computing the trace of $\widetilde{R}_\Gamma (f)$ in two different ways. It takes the following form (the LHS is the spectral side and the RHS is the geometric side):
\begin{equation} \label{TTF}
\sum_{\pi \in \Pi (\widetilde{G})} m (\pi , \widetilde{\pi} , \Gamma) \  \mathrm{trace} \ \widetilde{\pi} (f) = \sum_{\{ \delta \}} \vol (\Gamma^{\delta} \backslash G^{\delta}) \int_{G^{\delta} \backslash G} f(x^{-1} \delta x) d \dot{x}. 
\end{equation}
Here $\widetilde{\pi}$ is some extension of $\pi$ to a twisted representation of $\widetilde{G}$ and 
\begin{align*}
m (\pi , \widetilde{\pi} , \Gamma) &= \sum_{\widetilde{\pi} ' | G = \pi} \lambda (\widetilde{\pi} ' , \widetilde{\pi}) m (\widetilde{\pi} ') \\
&= \mathrm{trace}\left(\sigma | \Hom_G(\widetilde{\pi}, L^2(\Gamma \backslash G)) \right), 
\end{align*}
where $m (\widetilde{\pi} ')$ is the multiplicity 
of $\widetilde{\pi} '$ in $\widetilde{R}_\Gamma$ and $\lambda (\widetilde{\pi} ' , \widetilde{\pi}) \in \C^{\times}$ is the scalar s.t. for all $\delta \in \widetilde{G}$, we have 
$\widetilde{\pi} ' (\delta ) = \lambda (\widetilde{\pi} ' , \widetilde{\pi}) \widetilde{\pi} (\delta)$.\footnote{Note that $m (\pi , \widetilde{\pi} , \Gamma) \  \mathrm{trace} \ \widetilde{\pi} (f)$ does not depend on the chosen particular extension 
$\widetilde{\pi}$ but only on $\pi$.} Note that $\lambda (\widetilde{\pi} ' , \widetilde{\pi})$ is in fact a $p$-th root of unity.

The definition of $ \mathrm{trace} \ \widetilde{\pi} (f)$ depends on a choice of a Haar measure $dx$ on $G$. On the geometric side the volumes  
$\vol (\Gamma^{\delta} \backslash G^{\delta})$ depend on choices of Haar measures on the groups $G^{\delta}$. We will make precise 
choices later on. For the moment we just note that the measure $d\dot{x}$ on the quotient $G^{\delta} \backslash G$ is normalized by:
$$\int_{G} \phi (x) dx = \int_{G^{\delta} \backslash G} \int_{G^{\delta}} \phi (gx) dg d\dot{x}.$$

\subsection{Galois cohomology groups $H^1 (\sigma , \Gamma)$} \label{par:C}
Let $Z^1 (\sigma , \Gamma) = \{ \delta \in \widetilde{\Gamma} \; : \; \delta^p = \sigma\}$; it is invariant by conjugation by $\Gamma$. We denote by $H^1 (\sigma , \Gamma)$ 
the quotient of $Z^1 (\sigma , \Gamma)$ by the equivalence relation defined by conjugation by elements of $\Gamma$. We have similar definitions when $\Gamma$
is replaced by $G$. 

We will assume that 
\begin{equation} \label{condition}
H^1 (\sigma , G ) = \{1 \}.
\end{equation}
Note that if $\mathbb{E} = \RR^p,$ condition \eqref{condition} is always satisfied. Indeed: in that case $G = \G (\RR)^p$ and an element 
$$(g_1 , \ldots , g_p) \rtimes \sigma \in \widetilde{G}$$
belongs to $Z^1 (\sigma , G)$ if and only if $g_1 g_2 \cdots g_p = e$. But then there is an equality 
$$\sigma (g_1 , g_1 g_2 , \ldots , g_1 \cdots g_p )^{-1} (g_1 , g_1 g_2 , \ldots , g_1  \cdots g_p) = (g_1 , \ldots , g_p).$$
Equivalently, $(g_1 , \ldots , g_p) \rtimes \sigma$ is conjugated to $\sigma$ in $\widetilde{G}$ by some element in $G$. 

We furthermore note that $H^1 (\C / \R , \mathrm{SL}_n (\C)) = H^1 (\C / \R , \mathrm{Sp}_n (\C )) = \{1\}$, see e.g. \cite[Chap. X]{Serre}. 
Therefore, condition \eqref{condition} holds if $\G$ is a product of factors $\mathrm{SL}_n$ or $\mathrm{Sp}_n$ or of factors whose group of real points
is isomorphic to a complex Lie group viewed as a real Lie group. 

\subsection{} Under assumption \eqref{condition}, the map 
$$H^1 (\sigma , \Gamma) \to H^1 (\sigma , G)$$
necessarily has trivial image. In other words: if $\delta$ represents a class in $H^1 (\sigma , \Gamma)$ then $\delta$ is conjugate to $\sigma$ by some element of $G$. In particular,
we have
$$\vol (\Gamma^{\delta} \backslash G^{\delta}) = \vol (\Gamma^{\sigma} \backslash G^{\sigma}) \mbox{ and }  \int_{G^{\delta} \backslash G} 
f (x^{-1} \delta x) d\dot{x} =   \int_{G^{\sigma} \backslash G} 
f (x^{-1} \sigma x) d\dot{x}.$$
We may therefore write the geometric side of the twisted trace formula as: 
\begin{equation} \label{TTF2}
|H^1 (\sigma , \Gamma) | \vol (\Gamma^{\sigma} \backslash G^{\sigma}) \int_{G^{\sigma} \backslash G} 
f (x^{-1} \sigma x) d\dot{x} + \sum_{\substack{\{\delta \} \\ \delta \notin Z^1 (\sigma , \Gamma)}} \vol (\Gamma^{\delta} \backslash G^{\delta}) \int_{G^{\delta} \backslash G} 
f (x^{-1} \delta x) d\dot{x}.
\end{equation}

\subsection{Finite dimensional representations of $\widetilde{G}$} \label{def:SA}
Note that the complexification of $G$ may be identified with the complex points of $\mathrm{Res}_{\mathbb{E} / \RR} \mathbf{G}$, i.e. 
$\mathbf{G} (\C)^p$. Every complex finite dimensional $\sigma$-stable irreducible representation $(\widetilde{\rho} , F)$ of $\widetilde{G}$ can therefore be realized in a space $F = F_0^{\otimes p}$ where $(\rho_0 , F_0)$ is an irreducible complex linear representation of $\G (\C).$  The action of 
$G$ is defined by the tensor product action $\rho_0^{\otimes p}$ if $\mathbb{E} = \RR^p$ and by 
$\otimes_{i=1}^{p/2} (\rho_0 \otimes \bar{\rho}_0)$, where $\bar{\rho}_0$ is obtained by composing the complex conjugation 
in $\G (\C)$ by $\rho_0$, if $\mathbb{E} = \C^{p/2}$. In both cases, we choose the action of $\sigma$ on $F = F_0^{\otimes p}$ to be the cyclic permutation $A: x_1 \otimes \ldots \otimes x_p \mapsto x_p \otimes x_1 \otimes \ldots \otimes x_{p-1}$. Note that 
$$\mathrm{trace} (\sigma \; | \;  F ) = \dim F_0.$$

Let $\g$ be the Lie algebra of $\G.$ Say that $(\widetilde{\rho} , F)$ is {\it strongly twisted acyclic}  
if there is a positive constant $\eta$ depending only on $F$ such that: for every irreducible unitary $(\g , \widetilde{K})$-module $V$ for which 
$$\mathrm{trace} (\sigma \; | \;  C^j (\g (\C) , K , V\otimes F)) \neq 0$$
for some $j \leq \dim \G(\C),$ the inequality 
$$\Lambda_F - \Lambda_V \geq \eta$$
is satisfied.  Here $\Lambda_F$, resp. $\Lambda_V$, is the scalar by which the Casimir acts on $F$, resp. $V$.

Write $\nu$ for the highest weight of $F_0$. The following lemma can be proven analogously to \cite[Lemma 4.1]{BV}.

\begin{lem} 
Suppose that $\nu$ is not preserved by the Cartan involution $\theta.$  Then $\widetilde{\rho}$ is strongly twisted acyclic.
\end{lem}

\section{Lefschetz number and twisted analytic torsion}

Let $G$, $\sigma$ and $\Gamma$ be as in \S \ref{TTF1} and let  $(\tilde{\rho} , F)$ be a complex finite dimensional 
$\sigma$-stable irreducible representation of $\widetilde{G}$.  
We denote by $\g$ be the Lie algebra of $G$. 

\subsection{Twisted $(\g (\C) , K)$-cohomology and Lefschetz number} \label{lefschetznumber}
We can define an action of $\sigma$ on each cohomology group $H^i (\Gamma \backslash X , F)$ and thus define a {\it Lefschetz number}
$$\mathrm{Lef} (\sigma , \Gamma , F) = \sum_i (-1)^i \tr \ ( \sigma \; | \; H^i (\Gamma \backslash X , F)).$$

If $V$ is a $(\g , \widetilde{K})$-module, we have a natural action of $\sigma$ on the space of $(\g , K)$-cochains $C^\bullet (\g (\C) , K , V)$ which induces an action on the quotient $H^\bullet (\g (\C) , K , V)$. We denote by 
$$\mathrm{trace} \ (\sigma \; | \; H^\bullet (\g  , K , V))$$
the trace of the corresponding operator. We then define the {\it Lefschetz number} of $V$ by
$$\mathrm{Lef} (\sigma , V ) = \sum_i (-1)^i \mathrm{trace} (\sigma \; | \; H^i (\g  , K , V)).$$
If $F$ is a finite dimensional representation of $\widetilde{G}$ then $F \otimes V$ is still a $(\g  , \widetilde{K})$-module; we denote by $\mathrm{Lef}(\sigma , F , V)$ its Lefschetz number. 

Labesse \cite[Sect. 7]{Labesse} proves that there exists a compactly supported function $L_{\rho} \in C_c^{\infty} (\widetilde{G})$ such that for every 
essential admissible representation $(\widetilde{\pi} , V)$ of $\widetilde{G}$ one has
$$\mathrm{Lef} (\sigma , F , V) = \tr \ \widetilde{\pi} (L_{\rho}).$$
The function $L_{\rho}$ is called the {\it Lefschetz function} for $\sigma$ and $(\tilde{\rho}, F)$. 

We then have:
\begin{equation} 
\begin{split}
\tr \ \widetilde{R}_{\Gamma} (L_{\rho}) & = \sum_i (-1)^i \tr ( \sigma \; | \; H^i (\Gamma \backslash X , F)) \\
& = \mathrm{Lef} (\sigma , \Gamma , F) .
\end{split}
\end{equation}

\subsection{Twisted heat kernels}
Let
$$H_t^{\rho , i} \in \left[ C^{\infty} (G) \otimes \mathrm{End} (\wedge^i (\g / \k)^* \otimes F) \right]^{K \times K}$$
be the heat kernel for $L^2$-forms of degree $i$ with values in the bundle associated to $(\rho , F)$. Note that 
we have a natural action of $\sigma$ on $\wedge^i (\g / \k)^* \otimes F$; we denote by $A_{\sigma}$ the corresponding linear operator and let
$$h_t^{\rho, i, \sigma} : x \rtimes \sigma \mapsto \tr ( H_t^{\rho , i} (x) \circ A_{\sigma}) .$$

Eventually we shall apply the twisted trace formula to $h_t^{\rho, i, \sigma}$.  The heat kernel $H_t^{\rho , i}$ is not compactly supported.  However, it follows from 
\cite[Proposition 2.4]{BM} that it belongs to all Harish-Chandra Schwartz spaces $\mathcal{C}^q \otimes \mathrm{End} (\wedge^i (\g / \k)^* \otimes F)$, $q>0.$  This
is enough to ensure absolute convergence of both sides of the twisted trace formula.

\begin{lem} \label{L1}
Let $\widetilde{\pi}$ be an essential admissible representation of $\widetilde{G}$ and let $V$ be its associated $(\g  , \widetilde{K})$-module. We have:
$$ \tr \ \widetilde{\pi} (h_t^{\rho , i, \sigma}) = e^{t(\Lambda_V - \Lambda_F)} \tr (\sigma | [\wedge^i (\g / k )^* \otimes F \otimes V]^K).$$
\end{lem}
\begin{proof} It follows from the $K\times K$ equivariance of $H_t^{\rho , i}$ and Kuga's Lemma that relative to the splitting
$$\wedge^i (\g  / k )^* \otimes F \otimes V = [\wedge^i (\g  / k )^* \otimes F \otimes V]^K \oplus \left( [\wedge^i (\g / k )^* \otimes F \otimes V]^K \right)^{\perp},$$
we have:
$$\pi (H_t^{\rho , i}) = \left(
\begin{array}{cc}
e^{t(\Lambda_V - \Lambda_F)} \mathrm{Id} & 0 \\
0 & 0 
\end{array} \right).$$
We furthermore note that this decomposition is $\sigma$-invariant since $K$ is $\sigma$-stable. We conclude that we have:
$$\widetilde{\pi} (H_t^{\rho ,i}) := \int_{\G (\C)} (\pi (g)\circ A) \otimes (H_t^{\rho, i} (g) \circ A_{\sigma}) dg =  \left(
\begin{array}{cc}
e^{t(\Lambda_V - \Lambda_F)} A_{\sigma} & 0 \\
0 & 0 
\end{array} \right).$$
Here $\pi$ is the restriction of $\widetilde{\pi}$ to $G$ and $A$ is the intertwining operator between $\pi$ and $\pi \circ \sigma$ that determines $\widetilde{\pi}$. 

Now let $\{\xi_n \}_{n \in \N}$ and $\{e_j \}_{j=1, \ldots , m}$ be orthonormal bases of
$V$ and $\wedge^i (\g  / k )^* \otimes F$, respectively. Then we have:
\begin{equation*}
\begin{split}
\tr \ \widetilde{\pi} (H_t^{\rho , i}) & = \sum_{n=1}^{\infty} \sum_{j=1}^m \langle \widetilde{\pi} (H_t^{\rho ,i}) (\xi_n \otimes e_j ) , (\xi_n \otimes e_j ) \rangle \\
& = \sum_{n=1}^{\infty} \sum_{j=1}^m \int_{G} \langle (\pi (g) \circ A) \xi_n , \xi_n \rangle \langle (H_t^{\rho ,i} (g) \circ A_{\sigma}) e_j  , e_j  \rangle dg \\
& = \sum_{n=1}^{\infty} \int_{G} \langle (\pi (g) \circ A) \xi_n , \xi_n \rangle h_t^{\rho , i , \sigma} (g \rtimes \sigma) dg \\
& = \tr \ \widetilde{\pi} (h_t^{\rho , i, \sigma}).
\end{split}
\end{equation*}
The lemma follows.
\end{proof}

Denoting by $H_t^0 \in \left[ C^{\infty} (G) \otimes \mathrm{End} (\wedge^0 (\g / \k)^*) \right]^{K \times K}$ 
the heat kernel for $L^2$-functions on $X$, the following proposition follows from \cite[Proposition 5.3]{MuellerPfaff} and the definition of strong twisted acyclicity. 

\begin{prop} \label{P:larget}
Assume that $(\widetilde{\rho} , F)$ is strongly twisted acyclic. Then there exist positive constants $\eta$ and $C$ such that for every $x \in G$, $t \in (0, +\infty)$ and 
$i \in \{0 , \ldots , \dim X \}$, one has:
$$|h_t^{\rho , i , \sigma} (x \rtimes \sigma) | \leq C e^{-\eta t} H_t^0 (x).$$
\end{prop}

We define the kernel $k_t^{\rho, \sigma}$ by
$$k_t^{\rho , \sigma} (g) = \sum_i (-1)^i i h_t^{\rho , i , \sigma} (g);$$
it defines a function in $\mathcal{C}^q (\widetilde{G})$, for all $q>0$.

\subsection{Twisted analytic torsion}
The {\it twisted
analytic torsion} $T_{\Gamma \backslash X}^{\sigma} (\rho)$ is then defined by
\begin{equation} \label{TT}
\log T_{\Gamma \backslash X}^{\sigma} (\rho) = \frac12 \frac{d}{ds} {}_{\mid_{s= 0}} \left( \frac{1}{\Gamma (s)} \int_0^{\infty} t^{s-1} \left[ \tr \ \widetilde{R}_{\Gamma} (k_t^{\rho ,\sigma}) - 
\sum_i (-1)^i i \cdot \tr ( \sigma \; | \; H^i (\Gamma \backslash X , F)) \right] dt \right).
\end{equation}
Note that if $(\widetilde{\rho} , F)$ is strongly twisted acyclic each $\tr ( \sigma \; | \; H^i (\Gamma \backslash X , F))$ is trivial. From now on we will assume that $(\widetilde{\rho} , F)$ is strongly twisted acyclic.  
In particular, we have:
$$\log T_{\Gamma \backslash X}^{\sigma} (\rho) = \frac12 \frac{d}{ds} {}_{|s= 0} \left( \frac{1}{\Gamma (s)} \int_0^{\infty} t^{s-1} \tr \ \widetilde{R}_{\Gamma} (k_t^{\rho ,\sigma}) dt \right).$$

\subsection{Twisted $(\g , K)$-torsion} 
If $V$ is a $(\g , \widetilde{K})$-module and $F$ is a finite dimensional representation of $\widetilde{G}$ then $F \otimes V$ is still a $(\g , \widetilde{K})$-module; we define the 
{\it twisted $(\g , K)$-torsion} of $F \otimes V$ by 
\begin{equation*}
\begin{split}
\mathrm{Lef} {}' (\sigma , F , V) & = \sum_i (-1)^i i \ \mathrm{trace} (\sigma \; | \; C^i (\g , K , F \otimes V)) \\
& = \sum_i (-1)^i i \ \mathrm{trace} (\sigma \; | \; [\wedge^i (\g / \k)^*  \otimes F \otimes V]^K).
\end{split}
\end{equation*} 

\medskip
\noindent
{\it Remark.} We should explain the notation $\mathrm{Lef}'$. Given a group $G$ and a $G$-vector space $V$, we denote by $\det[1-V]$ the virtual $G$-representation
(that is to say, element of $K_0$ of the category of $G$-representations) defined by the alternating sum $\sum_i (-1)^i [\wedge^i V]$ of exterior powers. 
This is multiplicative in an evident sense:
\begin{equation} \label{add1} 
\det[1 - V \oplus W] = \det[1-V] \otimes \det[1-W].
\end{equation} 
Now given $g \in G$, the derivative $\frac{d}{dt} {}_{|_{t=1}} \det (t1 - g)$
is equal to the character of $g$ acting on $\sum_{i} (-1)^i i \wedge^i V$. We therefore define $\detp[1-V] = \sum_{i} (-1)^i i \wedge^i V$.

Considering the virtual $\widetilde{K}$-representation $\detp [1- (\g  / \k)^*]$ we have: 
\begin{equation} \label{Lef'}
\mathrm{Lef} {}' (\sigma , F , V) = \mathrm{trace} \left( \sigma \; | \; \left[ \detp [1 - (\g  / \k)^*] \otimes F \otimes V\right]^K \right).
\end{equation}
This explains our notation $\mathrm{Lef}'$ for the twisted $(\g  , K)$-torsion. 

For future reference we note that we have:
\begin{equation} \label{add} 
\detp[1 - V \oplus W] = \detp[1-V] \otimes \det[1-W] \oplus \det[1-V] \otimes \detp[1-W].
\end{equation} 

We also note that  Labesse's proof of the existence of $L_{\rho}$ can be immediately modified to get a function 
$L_{\rho}' \in C_c^{\infty} (\widetilde{G})$ such that for every 
essential admissible representation $(\widetilde{\pi} , V)$ of $\widetilde{G}$ one has
$$\mathrm{Lef}' (\sigma , F , V) = \tr \ \widetilde{\pi} (L_{\rho}').$$

\subsection{} It follows from Lemma \ref{L1} that the spectral side
of the twisted trace formula evaluated in $k_t^{\rho , \sigma}$ is
\begin{equation}
\tr \ \widetilde{R}_{\Gamma} (k_t^{\rho ,\sigma}) = 
\sum_{\pi \in \Pi (\widetilde{G})} m (\pi , \widetilde{\pi} , \Gamma ) \mathrm{Lef}' (\sigma , F , V_{\pi}) e^{t(\Lambda_{\pi} - \Lambda_F)},
\end{equation}
where $V_{\pi}$ is the $(\g  , \widetilde{K})$-module associated to the extension $\widetilde{\pi}$.

\section{$L^2$-Lefschetz number, $L^2$-torsion and limit formulas}

Let $G$, $\sigma$ and $\Gamma$ be as in \S \ref{TTF1} and let  $(\tilde{\rho} , F)$ be as in the preceding sections. 
Let $f \in C_c^{\infty} (\widetilde{G})$. 

\subsection{} 
Given $g \in G$ we define $r(g) = \mathrm{dist}(gK,eK)$ with respect to the Riemannian symmetric distance of $X=G/K$. 
We extend $r$ to $G \rtimes \langle \sigma \rangle$ by setting 
$r(g \rtimes \sigma) = r(g)$. Note that $r(gg') \leq r(g) + r(g')$. 

Now given $x \in G$ we set $\ell (x) =  \inf \{ r(gxg^{-1}) \; : \; g \in G \}$; it only depends on the conjugacy class of $x$ in $G$. Recall that the {\it injectivity radius of $\Gamma$},
$$r_{\Gamma} = \frac12 \inf \{ \ell (\gamma ) \; : \; \gamma \in \Gamma - \{ e \} \},$$
is strictly positive. Now if $\delta \in \widetilde{\Gamma}$ with $\delta \notin  Z^1 (\sigma , \Gamma)$ then $\delta^p \in \Gamma - \{ 1\}$ and therefore $\ell (\delta^p) \geq 2r_{\Gamma}$. 
In particular for any $x \in G$ we have:
\begin{equation} \label{twistedinjrad}
2r_\Gamma \leq \ell (\delta^p) \leq r(x\delta^px^{-1}) = r( (x \delta x^{-1} ) \cdots (x \delta x^{-1})) \leq p \cdot r(x \delta x^{-1}).
\end{equation}

\begin{lem} \label{Lcount}
There exist constants $c_1,c_2 >0$, depending only on $G$, such that for any $x \in G$, we have:
$$N(x; R ) := |\{ \delta \in \widetilde{\Gamma} \; : \; \delta \notin  Z^1 (\sigma , \Gamma) \mbox{ and } r(x \delta x^{-1}) \leq R \} | \leq c_1 p^d r_\Gamma^{-d} e^{c_2 R} ,$$
where $d$ is the dimension of $X$.
\end{lem}
\begin{proof}
It follows from \eqref{twistedinjrad} that it suffices to prove the lemma for $R \geq 2r_\Gamma/p$. Set $\varepsilon := r_\Gamma/p$. By definition of $r(x \delta x^{-1})$ we have 
$B(\gamma (\sigma (x)) , \varepsilon ) \subset B(x, R+\varepsilon)$, for all $\delta =  \gamma \rtimes \sigma \in \widetilde{\Gamma}$ with $\delta \notin  Z^1 (\sigma , \Gamma)$ and 
$r(x \delta x^{-1}) \leq R$. Now, since $\varepsilon < r_\Gamma$, the balls $B(\gamma (\sigma (x)) , \varepsilon )$, $\gamma \in \Gamma$, are all disjoint of the same volume. We conclude that  
$$N(x;R) \cdot \mathrm{vol} B (\sigma (x) , \varepsilon ) \leq \mathrm{vol} B( x , R + \varepsilon ) \leq \mathrm{vol} B (x, 2R).$$
We conclude using standard estimates on volumes of balls (see e.g. \cite[Lemma 7.21]{7samurai} for more details).
\end{proof}

\begin{prop} \label{WLF}
Let $\{\Gamma_n \}$ be a sequence of finite index $\sigma$-stable subgroups of $\Gamma$. Assume that for any 
$\delta \in \widetilde{\Gamma}$ with $\delta \notin Z^1 (\sigma , \Gamma)$ we have:
$$\lim_{n \to \infty} \frac{| \{ \gamma \in \Gamma_n \backslash \Gamma \; : \; \gamma \delta \gamma^{-1} \in \widetilde{\Gamma}_n \}|}{[\Gamma^\sigma : \Gamma_n^\sigma]} = 0.$$
Then 
$$\frac{\tr \ \widetilde{R}_{\Gamma_n} (f)}{|H^1 (\sigma , \Gamma_n) | \vol (\Gamma_n^{\sigma} \backslash G^{\sigma})} \to  \int_{G^{\sigma} \backslash G} f(x^{-1} \sigma x) d\dot{x}.$$
\end{prop}
\begin{proof} It follows from \eqref{TTF2} that we have:
\begin{equation*} 
\begin{split}
\tr \ \widetilde{R}_{\Gamma_n} (f ) & = |H^1 (\sigma , \Gamma_n) | \vol (\Gamma_n^{\sigma} \backslash G^{\sigma}) \int_{G^{\sigma} \backslash G} 
f(x^{-1} \sigma x) d\dot{x} + \sum_{\substack{\{\delta \}_{\Gamma_n} \\ \delta \notin Z^1 (\sigma , \Gamma_n)}} \vol (\Gamma_n^{\delta} \backslash G^{\delta}) \int_{G^{\delta} \backslash G} f  (x^{-1} \delta x) d\dot{x} \\
& = |H^1 (\sigma , \Gamma_n) | \vol (\Gamma_n^{\sigma} \backslash G^{\sigma}) \int_{G^{\sigma} \backslash G} 
f(x^{-1} \sigma x) d\dot{x} + \sum_{\substack{\{\delta \}_\Gamma \\ \delta \notin Z^1 (\sigma , \Gamma)}} c_{\Gamma_n} (\delta ) \vol (\Gamma^{\delta} \backslash G^{\delta}) \int_{G^{\delta} \backslash G} f  (x^{-1} \delta x) d\dot{x}
\end{split}
\end{equation*}
where 
$$c_{\Gamma_n} (\delta) = | \{ \gamma \in \Gamma_n \backslash \Gamma \; : \; \gamma \delta \gamma^{-1} \in \widetilde{\Gamma}_n \}|.$$
Since $f$ is compactly supported, the last sum above is finite: choosing $R>0$ so that the support of $f$ is contained in 
$$B_R = \{ g \rtimes \sigma \in \widetilde{G} \; : \; r(g) \leq R\}$$
we may restrict the sum on the right side of the above equation to $\delta$ that are contained in $B_R$. It follows from Lemma \ref{Lcount} that the corresponding sum if finite.  
\end{proof}

\subsection{} Let $\{ \Gamma_n \}$ be a normal chain with $\cap_n \Gamma_n = \{ 1 \}.$  Then 
$r_{\Gamma_n} \to \infty$ as $n \to \infty$. The following lemma implies that the hypotheses of Proposition \ref{WLF} are satisfied
for $\{ \Gamma_n \}$.

\begin{lem} \label{normalchain}
Let $\delta \notin Z^1 (\sigma , \Gamma)$. If $r_{\Gamma_n} \geq \ell (\delta^p)$, then we have: 
$$| \{ \gamma \in \Gamma_n \backslash \Gamma \; : \; \gamma \delta \gamma^{-1} \in \widetilde{\Gamma}_n \}|= 0.$$
\end{lem}
\begin{proof}
Indeed 
$$c_{\Gamma_n} (\delta)  \leq | \{ \gamma \in \Gamma_n \backslash \Gamma \; : \; \gamma \delta^p \gamma^{-1} \in \Gamma_n \} |$$
and if $\delta \notin Z^1 (\sigma , \Gamma)$ then $\delta^p \neq 1$. Finally, if 
$\gamma \delta^p \gamma^{-1} \in \Gamma_n$ is non-trivial then $\ell (\delta^p) \geq r_{\Gamma_n}$.
\end{proof}

\subsection{$L^2$-Lefschetz number} Proposition \ref{WLF} motivates the following definition of the {\it $L^2$-Lefschetz number}
associated to the triple $(G, \sigma, \rho)$:
$$\mathrm{Lef}^{(2)} (\sigma, X , F) = SO_{e \rtimes \sigma}(L_{\rho}) = \int_{G^{\sigma} \backslash G} L_\rho (x^{-1} \sigma x) d\dot{x}.$$

\begin{cor} \label{C:47}
Let $\{ \Gamma_n \}$ be a normal chain of finite index $\sigma$-stable subgroups of $\Gamma$ with $\cap_n \Gamma_n = \{ 1 \}$. Then we have:
$$\frac{\mathrm{Lef}(\sigma , \Gamma_n , F)}{|H^1 (\sigma , \Gamma_n) | \vol (\Gamma_n^{\sigma} \backslash G^{\sigma})} \to 
\mathrm{Lef}^{(2)} (\sigma, X , F) .$$
\end{cor}
\begin{proof} Apply Proposition \ref{WLF} (and Lemma \ref{normalchain}) to the Lefschetz function $L_{\rho}$. 
\end{proof}

\subsection{Twisted $L^2$-torsion} Analogously we define the {\it twisted $L^2$-torsion} 
$T_{\Gamma \backslash X}^{(2) \sigma} (\rho) \in \R^+$ by 
\begin{equation}
\log T_{\Gamma \backslash X}^{(2) \sigma} (\rho) = |H^1 (\sigma , \Gamma) | \vol (\Gamma^{\sigma} \backslash G^{\sigma}) t_X^{(2) \sigma} (\rho) 
\end{equation}
where $t_X^{(2) \sigma} (\rho)$ --- which depends only on the symmetric space $X$, the involution $\sigma,$ and the finite dimensional representation
$\rho$ --- is defined by
\begin{equation} \label{t}
t_X^{(2) \sigma} (\rho)  = \frac12 \frac{d}{ds} {}_{|s=0} \left( \frac{1}{\Gamma (s)} \int_0^{\infty} \int_{G^{\sigma} \backslash G} 
k_t^{\rho , \sigma} (x^{-1} \sigma x) d\dot{x} t^{s-1} dt \right) .
\end{equation}

Note that $k_t^{\rho , \sigma}$ is not compactly supported and that we have to prove that the RHS of \eqref{t}
is indeed well defined. Recall however that $k_t^{\rho , \sigma}$ belongs to $\mathcal{C}^q (\widetilde{G})$. Lemma \ref{Lcount} therefore implies that the series 
$$\sum_{\delta \in \widetilde{\Gamma}} k_t^{\rho , \sigma} (x^{-1} \delta x) $$
converges absolutely and locally uniformly. This implies that the integral of this series along a (compact) fundamental domain $D$ for the action of $\Gamma$ on $G$ is absolutely convergent. Restricting the sum to the $\delta$'s that belong to the (twisted) $\Gamma$-conjugacy class of $\sigma$ we conclude in particular that, for every positive $t$, the integral
$$\int_{G^{\sigma} \backslash G} k_t^{\rho , \sigma} (x^{-1} \sigma x) d\dot{x} = \frac{1}{\mathrm{vol} (\Gamma^{\sigma} \backslash G^\sigma)} \int_{D} \sum_{\delta \in \{ \sigma \}} k_t^{\rho , \sigma} (x^{-1} \delta x) dx$$
is absolutely convergent. We postpone the proof of the fact that \eqref{t} is indeed well defined until sections 6 and 7 where we will explicitly compute $t_X^{(2) \sigma}
(\rho)$. In the course of the computations we will also prove the following lemma.

\begin{lem} \label{Lkt}
There exist constants $C,c>0$ such that 
\begin{equation}
\left| \int_{G^{\sigma} \backslash G} k_t^{\rho , \sigma} (x^{-1} \sigma x) d\dot{x} \right| \leq Ce^{-ct}, \quad t \geq 1.
\end{equation}
\end{lem}

Granted this we conclude this section by the proof of the following `limit multiplicity theorem'.

\begin{thm} \label{twistedtorsionlimitmultiplicity}
Assume that $(\widetilde{\rho} , F)$ is strongly twisted acyclic. Let $\{\Gamma_n \}$ be a sequence of finite index $\sigma$-stable subgroups of $\Gamma$. Assume that
there exists a constant $A$ s.t. for every 
$\delta \in \widetilde{\Gamma}$ with $\delta \notin Z^1 (\sigma , \Gamma)$ the sequence 
$$\left(\frac{| \{ \gamma \in \Gamma_n \backslash \Gamma \; : \; \gamma \delta \gamma^{-1} \in \widetilde{\Gamma}_n \}|}{[\Gamma^\sigma : \Gamma_n^\sigma]} \right)_{n \geq 0}$$
remains uniformly bounded by $A$ and converges to $0$ as $n$ tends to infinity.\footnote{Note that, in the untwisted case, the condition $$\frac{| \{ \gamma \in \Gamma_n \backslash \Gamma \; : \; \gamma \delta \gamma^{-1} \in \Gamma_n \}|}{[\Gamma : \Gamma_n]} \to 0$$ is equivalent to the BS-convergence of the {\it compact} quotients $\Gamma_n \backslash X$ towards the 
symmetric space $X.$ See \cite{7samurai}.}
Then 
$$\frac{\log T_{\Gamma_n \backslash X}^{\sigma} (\rho)}{|H^1 (\sigma , \Gamma_n ) | \vol (\Gamma_n^{\sigma} \backslash G^{\sigma})} \to t_X^{(2) \sigma} (\rho) .$$
\end{thm}
\begin{proof} Since $k_t^{\rho , \sigma} \in \mathcal{C}^q (\widetilde{G})$, for all $q>0$, we still have:
\begin{multline*}
\tr \ \widetilde{R}_{\Gamma_n} (k_t^{\rho ,\sigma}) = |H^1 (\sigma , \Gamma_n) | \vol (\Gamma_n^{\sigma} \backslash G^{\sigma}) \int_{G^{\sigma} \backslash G} 
k_t^{\rho ,\sigma} (x^{-1} \sigma x) d\dot{x} \\
+ \sum_{\substack{\{\delta \}_\Gamma \\ \delta \notin Z^1 (\sigma , \Gamma)}} c_{\Gamma_n} (\delta ) \vol (\Gamma^{\delta} \backslash G^{\delta}) \int_{G^{\delta} \backslash G} k_t^{\rho ,\sigma} (x^{-1} \delta x) d\dot{x}.
\end{multline*}
Note that at this point it is not clear that the sum on the right (absolutely) converges. This is however indeed the case: it first follows from \eqref{twistedinjrad} that if $\delta \notin  Z^1 (\sigma , \Gamma)$ and $x \in G$ we have $r(x \delta x^{-1}) \geq 2r_\Gamma/p$. 
Now recall from \cite[Lemma 3.8]{BV} or \cite[Proposition 3.1 and (3.14)]{MuellerPfaff} that we have, for $t \in (0, 1]$,
\begin{equation} \label{kt}
|k_t^{\rho , \sigma} (x^{-1} \delta x)| \leq  C t^{-d} \exp \left( -c \frac{r(x \delta x^{-1})^2}{t} \right) \leq C e^{-c'/t} \exp \left( - c'' r(x \delta x^{-1})^2\right).
\end{equation}
(Here $c'$ depends on $r_\Gamma$.) From this and Lemma \ref{Lcount}, it follows that the geometric side of the trace formula evaluated in $k_t^{\rho ,\sigma}$ indeed absolutely converges. 
Moreover, it follows from \eqref{kt} together with our uniform boundedness assumption that we have:
\begin{equation}
\tr \ \widetilde{R}_{\Gamma_n} (k_t^{\rho ,\sigma}) = |H^1 (\sigma , \Gamma_n) | \vol (\Gamma_n^{\sigma} \backslash G^{\sigma}) \int_{G^{\sigma} \backslash G} 
k_t^{\rho , \sigma} (x^{-1} \sigma x) d\dot{x}  + O (e^{-c' /t})
\end{equation}
for $0 < t \leq 1$. It follows that 
$$\int_0^1 t^{s-1} \left( \frac{\tr \ \widetilde{R}_{\Gamma_n} (k_t^{\rho ,\sigma}) }{|H^1 (\sigma , \Gamma_n) | \vol (\Gamma_n^{\sigma} \backslash G^{\sigma})} - 
\int_{G^{\sigma} \backslash G} 
k_t^{\rho ,\sigma} (x^{-1} \sigma x) d\dot{x} \right) dt$$
is holomorphic is $s$ in a half-plane containing $0$, so 
\begin{multline*}
\frac12 \frac{d}{ds} {}_{|s=0}  \frac{1}{\Gamma (s)} \int_0^{+\infty} t^{s-1} \left( \frac{\tr \ \widetilde{R}_{\Gamma_n} (k_t^{\rho ,\sigma}) }{|H^1 (\sigma , \Gamma_n) | \vol (\Gamma_n^{\sigma} \backslash G^{\sigma})} - 
\int_{G^{\sigma} \backslash G} 
k_t^{\rho ,\sigma} (x^{-1} \sigma x) d\dot{x} \right) dt \\
= \int_0^{+\infty} \left( \frac{\tr \ \widetilde{R}_{\Gamma_n} (k_t^{\rho ,\sigma}) }{|H^1 (\sigma , \Gamma_n) | \vol (\Gamma_n^{\sigma} \backslash G^{\sigma})} - 
\int_{G^{\sigma} \backslash G} 
k_t^{\rho ,\sigma} (x^{-1} \sigma x) d\dot{x} \right) \frac{dt}{t}.
\end{multline*}
Now it follows from Proposition \ref{P:larget} that there exists some positive $\eta$ such that $|k_t^{\rho , \sigma} (x \rtimes \sigma )| \ll e^{- \eta t} H_t^0 (x)$. In particular $|k_t^{\rho , \sigma} (x \rtimes \sigma )| \ll e^{- \eta t} H_1^0 (x)$ if $t \geq 1$ and we have: 
$$\frac{|\tr \ \widetilde{R}_{\Gamma_n} (k_t^{\rho ,\sigma})|}{|H^1 (\sigma , \Gamma_n) | \vol (\Gamma_n^{\sigma} \backslash G^{\sigma})} \ll e^{-\eta t} \sum_{
\{\delta \}_\Gamma } \int_{G^{\delta} \backslash G} H_1^{0} (x^{-\sigma} \delta x) d\dot{x}.$$ 
The above sum is absolutely convergent and independent of $t$ and $n$, implying that  
$$\frac{|\tr \ \widetilde{R}_{\Gamma_n} (k_t^{\rho ,\sigma})|}{|H^1 (\sigma , \Gamma_n) | \vol (\Gamma_n^{\sigma} \backslash G^{\sigma})} \ll e^{-\eta t}$$
where the implicit constant does not depend on $n$. Using Lemma \ref{Lkt}, we conclude that both 
$$\int_1^{+\infty} \frac{\tr \ \widetilde{R}_{\Gamma_n} (k_t^{\rho ,\sigma}) }{|H^1 (\sigma , \Gamma_n) | \vol (\Gamma_n^{\sigma} \backslash G^{\sigma})}\frac{dt}{t} \mbox{ and } \int_1^{+\infty} \int_{G^{\sigma} \backslash G} 
k_t^{\rho ,\sigma} (x^{-1} \sigma x) d\dot{x}  \frac{dt}{t}$$
are absolutely convergent uniformly in $n.$ We are therefore reduced to proving that 
$$\frac{\tr \ \widetilde{R}_{\Gamma_n} (k_t^{\rho ,\sigma}) }{|H^1 (\sigma , \Gamma_n) | \vol (\Gamma_n^{\sigma} \backslash G^{\sigma})} - 
\int_{G^{\sigma} \backslash G} 
k_t^{\rho ,\sigma} (x^{-1} \sigma x) d\dot{x} \to 0$$
uniformly in $t$ when $t$ belongs to a compact subinterval of $[0, +\infty)$. But this follows from the proof of Proposition \ref{WLF} and the fact that for every $\delta \in \Gamma$ 
the sequence 
$$\left(\frac{| \{ \gamma \in \Gamma_n \backslash \Gamma \; : \; \gamma \delta \gamma^{-1} \in \widetilde{\Gamma}_n \}|}{[\Gamma^\sigma : \Gamma_n^\sigma]} \right)_{n \geq 0}$$
remains uniformly bounded by $A$ and converges to $0$ as $n$ tends to infinity. 
\end{proof}

\medskip
\noindent
{\it Remark.} Though reminiscent of a natural condition in the non-twisted case, we do not know how to check the condition on $c_{\Gamma_n} (\delta )$ stated in Theorem \ref{twistedtorsionlimitmultiplicity} for any non-arithmetic $\Gamma.$  However, we show in the next section that it holds for many sequences of congruence subgroups of arithmetic groups.

\section{Bounding the growth of $c_{\Gamma_n}(\delta)$ for congruence subgroups of arithmetic groups}

We begin with a lemma to be used heavily in the proof of Proposition \ref{boundinggrowth}, which bounds the growth of $c_{\Gamma_p}(\delta).$

\begin{lem}[Rational points of inner forms] \label{innerforms}
Let $\mathbf{P}$ be any connected, affine algebraic group over a finite field $k.$  If $\mathbf{P}'$ over $k$ is any inner form of $\mathbf{P},$ then $|\mathbf{P}(k)| = |\mathbf{P}'(k)|.$ 
\end{lem}
\begin{proof}
We first reduce to the smooth case.  Denote by $\mathbf{P}_{\rm red}$ the underlying reduced scheme of $\mathbf{P}$. We have an isomorphism $(\mathbf{P}_{\rm red}  \times_k \mathbf{P}_{\rm red})_{\rm red} 
\stackrel{\sim}{\to} (\mathbf{P} \times_k \mathbf{P})_{\rm red}$, see \cite[Chap. I, Cor. 5.1.8]{EGA}. Since the field $k$ is perfect, we moreover have $(\mathbf{P}_{\rm red}  \times_k \mathbf{P}_{\rm red})_{\rm red} 
\stackrel{\sim}{\to} \mathbf{P}_{\rm red}  \times_k \mathbf{P}_{\rm red}$ so that the group law $\mathbf{P} \times \mathbf{P} \to \mathbf{P}$ induces a group law on $\mathbf{P}_{\rm red}.$ Hence, $\mathbf{P}_{\rm red}$ is a closed subgroup $k$-scheme of $\mathbf{P}.$  Since every reduced finite type scheme over a perfect field is smooth over a dense open subscheme, the standard homogeneity argument implies the group $\mathbf{P}_{\rm red}$ is smooth.  But $\mathbf{P}(k) = \mathbf{P}_{\mathrm{red}}(k), \mathbf{P}'_{\mathrm{red}}(k) = \mathbf{P}'(k),$ and $\mathbf{P}_{\mathrm{red}}$ is an inner form of $\mathbf{P}'_{\mathrm{red}}.$  We may therefore assume that $\mathbf{P}$ and $\mathbf{P}'$ are smooth.
  
Let $\mathbf{P}$ and $\mathbf{P}'$ have respective unipotent radicals $\mathbf{U}, \mathbf{U}'$ and respective reductive quotients $\mathbf{R}, \mathbf{R}'.$  Because $H^1(k,\mathbf{U}) = H^1(k,\mathbf{U}') = 0,$ there are exact sequences of finite groups
\begin{align*}
1 \rightarrow \mathbf{U}(k) &\rightarrow \mathbf{P}(k) \rightarrow \mathbf{R}(k) \rightarrow 1 \\
1 \rightarrow \mathbf{U}'(k) &\rightarrow \mathbf{P}'(k) \rightarrow \mathbf{R}'(k) \rightarrow 1.
\end{align*}
Because $\mathbf{U}$ and $\mathbf{U}'$ are forms, their dimensions are equal, say to $d.$  Furthermore, all unipotent groups over the perfect field $k$ are split unipotent.  Therefore, applying the vanishing of $H^1$ to filtrations of $\mathbf{U},\mathbf{U}'$ by subgroups whose successive quotients are $\mathbb{G}_a,$ we find that
$$|\mathbf{U}(k)| = |k|^d = |\mathbf{U}'(k)|.$$
Moreover, because $\mathbf{P}, \mathbf{P}'$ are inner forms, so are $\mathbf{R}, \mathbf{R}'.$  But $H^1(k, \mathrm{Inn}(\mathbf{R})) = 0$ since $\mathrm{Inn}(\mathbf{R})$ is connected.  Therefore, $\mathbf{R} \cong \mathbf{R}'$ over $k.$  In particular,
$$|\mathbf{R}(k)| = |\mathbf{R}'(k)|.$$
The result follows.
\end{proof}

Let $\G / O_{F,S}$ be a semisimple group, where $O_{F,S}$ denotes the ring of $S$-integers in a number field $F.$  For ease of exposition, we also assume that $\G$ is simply connected.  Let $E/F$ be a cyclic Galois extension with $\Gal(E/F) = \langle \sigma \rangle.$  Let $\mathbf{H} \subset \G$ be a connected algebraic subgroup smooth over $O_{F,S}.$  We fix integral structures so we may speak of $\G(O_F), \G(O_E)$ and $\mathbf{H}(O_F), \mathbf{H}(O_E).$ 

Fix a finite index, $\sigma$-stable subgroup $\Gamma \subset \G(O_E).$  Let $\mathfrak{p} \subset O_F$ be a prime ideal.  Let
\begin{equation} \label{congp}
\Gamma_\mathfrak{p}= \{ \gamma \in \Gamma : \gamma \in \mathbf{H}(O_E / \mathfrak{p}O_E) \text{ mod } \mathfrak{p} \}.
\end{equation}
Every $\Gamma_\mathfrak{p} \subset \G(E_\R)$ is a Galois-stable lattice. 

\medskip
\noindent
{\it Remark.} When $\mathbf{H}$ is trivial, the group $\Gamma_\mathfrak{p}$ is just the usual level $\mathfrak{p}$ congruence subgroup $\Gamma (\mathfrak{p})$. When $\mathbf{H}$ is the Borel subgroup of $\G$, the group $\Gamma_\mathfrak{p}$ is usually denoted by $\Gamma_0 (\mathfrak{p})$.

\medskip

\begin{prop} \label{boundinggrowth}
Suppose $\mathbf{H}_{\overline{F}}$ does not contain any normal subgroup of $\G_{\overline{F}}.$  Then the hypotheses of Proposition \ref{WLF} and Theorem \ref{twistedtorsionlimitmultiplicity} hold for the sequence $\Gamma_\mathfrak{p}.$  Namely,
\begin{itemize}
\item
There is a uniform upper bound
$$\frac{c_{\Gamma_\mathfrak{p}}(\delta)}{[\Gamma^{\sigma}:\Gamma_\mathfrak{p}^{\sigma}]} = \frac{|\{ \gamma \in \Gamma_\mathfrak{p} \backslash \Gamma: \gamma \delta \gamma^{-1}\in \widetilde{\Gamma}_\mathfrak{p}   \}|}{[\Gamma^{\sigma}:\Gamma_\mathfrak{p}^{\sigma}]} \leq C$$ 
for some constant $C$ depending only on $\G, \mathbf{H},$ and $\Gamma.$  

\item
For every $\delta,$
$$\lim_{\mathfrak{p} \rightarrow \infty} \frac{c_{\Gamma_\mathfrak{p}}(\delta)}{[\Gamma^{\sigma}:\Gamma_\mathfrak{p}^{\sigma}]} = 0.$$
\end{itemize}
\end{prop}

\begin{proof} We abuse notation and write $\delta = \delta \rtimes \sigma$, so that now $\delta$ belongs to $\Gamma$.  Let $$\Gamma_{\mathfrak{p}}^\ast = \{ \gamma \in \G(O_E) : \gamma \in \mathbf{H}(O_E / \mathfrak{p}O_E) \}.$$ 
For $\mathfrak{p}$ sufficiently large, the inclusion
$$\Gamma_{\mathfrak{p}} \backslash \Gamma \rightarrow \Gamma_\mathfrak{p}^\ast \backslash \G(O_E)$$
is a $\sigma$-equivariant isomorphism.  For such $\mathfrak{p},$ we may therefore identify
\begin{align*}
c_{\Gamma_\mathfrak{p}}(\delta) &= | \mathrm{Fix} (\delta \sigma | \Gamma_\mathfrak{p} \backslash \Gamma ) | \\
&= | \mathrm{Fix}(\delta \sigma | \Gamma_\mathfrak{p}^\ast \backslash \G(O_E)) | \\
&=| \mathrm{Fix}(\delta \sigma | \mathbf{H}(O_E/ \mathfrak{p}O_E) \backslash \G(O_E/ \mathfrak{p}O_E)) | \\ 
&= \tr \left( \delta \sigma | \mathrm{Ind}_{\mathbf{H}(O_E/\mathfrak{p}O_E)}^{\G(O_E/ \mathfrak{p}O_E)} 1 \right).
\end{align*}
Let $k = O_F / \mathfrak{p}.$  In the case where $O_E/\mathfrak{p} = k \times k,$ 
$$c_{\Gamma_\mathfrak{p}}(\delta) = \tr \left( \mathrm{Norm}(\delta) | \mathrm{Ind}_{\mathbf{H}(k)}^{\G(k)} 1 \right);$$
this follows, for example, from the identity
$$\tr \left( A_1 \otimes A_2 \circ  \text{cyclic permutation} | V^{\otimes 2} \right) = \tr \left( A_1 \circ A_2 | V \right)$$
valid for arbitrarily endomorphisms $A_i$ of a finite dimensional vector space $V.$  So in this case,
$$c_{\Gamma_\mathfrak{p}}(\delta) = | \mathrm{Fix}\left( \mathrm{Norm}(\delta) | \mathbf{H}(k) \backslash \G(k) \right)| \leq [\G(k): \mathbf{H}(k)] = [\Gamma^{\sigma}:\Gamma_\mathfrak{p}^{\sigma}].$$

We turn to the more interesting case where $O_E/\mathfrak{p}O_E = k'$ is a finite field extension of $k.$  The quantity $c_{\Gamma_\mathfrak{p}}(\delta)$ can be expressed explicitly as a sum over twisted conjugacy classes:
\begin{equation}\label{traceofinduced}
\tr \left( \delta \sigma | \mathrm{Ind}_{\mathbf{H}(O_E/\mathfrak{p}O_E)}^{\G(O_E/\mathfrak{p}O_E)} 1 \right) = \frac{|\mathbf{Z}_{\delta \rtimes \sigma}(k)|}{|\mathbf{H}(k')|} \times \left| \mathbf{H}(k') \cap \{ \G(k')\text{-twisted conjugacy class of $\delta$}\} \right|.
\end{equation}   
In \eqref{traceofinduced}, $\mathbf{Z}_{\G,\delta \rtimes \sigma}$ denotes the twisted centralizer in $R_{k'/k}\G$ of $\delta.$  The set 
$$\mathbf{H}(k') \cap  \{ \G(k')\text{-twisted conjugacy class of $\delta$} \}$$ 
is invariant under twisted conjugation and so decomposes as a finite disjoint union of $\mathbf{H}(k')$-twisted conjugacy classes.  Let $\{ y \}_{y \in I}$ be a full set of representatives for these conjugacy classes.

All of the $\mathbf{H}(k)$-conjugacy classes $[\mathrm{Norm}(y)], y \in I,$ are $\mathbf{G}(k)$-conjugate to $\mathrm{Norm}(\delta).$  Furthermore, the fiber over the conjugacy class $[\mathrm{Norm}(y)]$ under the norm map is in bijection with $H^1(\sigma, \mathbf{Z}_{\mathbf{H}, y \rtimes \sigma}(k))$ \cite[$\S 4$ Lemma 4.2]{Lan}.  By the vanishing of $H^1$ of connected groups over finite fields,
$$H^1(\sigma,\mathbf{Z}_{y \rtimes \sigma}(k)) \cong H^1(\sigma,\pi_0(\mathbf{Z}_{\mathbf{H},y \rtimes \sigma})(k)).$$  

Observe that
\begin{align} \label{firstinequality}
\frac{|\mathbf{Z}_{\mathbf{G},\delta \rtimes \sigma}(k)|}{|\mathbf{H}(k')|} \times & \left| \mathbf{H}(k') \cap \{ \G(k')\text{-twisted conjugacy class of $\delta$}\} \right| \nonumber \\
&= |\mathbf{Z}_{\mathbf{G},\delta \rtimes \sigma}(k)| \times \sum_{y \in I} \frac{| \{ \mathbf{H}(k')\text{-twisted conjugacy class of } y\} | }{|\mathbf{H}(k')|} \nonumber \\
&= |\mathbf{Z}_{\mathbf{G},\delta \rtimes \sigma}(k)| \times \sum_{y \in I} \frac{1}{| \mathbf{Z}_{\mathbf{H}, y \rtimes \sigma}(k) |}.
\end{align}

The groups $\mathbf{Z}_{\mathbf{H}, y \rtimes \sigma}$ and $\mathbf{Z}_{\mathbf{H},\mathrm{Norm}(y)}$ are inner forms \cite[$\S 1$ Lemma 1.1]{AC}.  By Lemma \ref{innerforms}, 
\begin{equation} \label{secondinequality}
\frac{1}{|\mathbf{Z}_{\mathbf{H},y \rtimes \sigma}(k)|} \leq |\pi_0(\mathbf{Z}_{\mathbf{H},\mathrm{Norm}(y)})(k)| \times \frac{1}{|\mathbf{Z}_{\mathbf{H},\mathrm{Norm}(y)}(k)|}.
\end{equation}
Let $M_{\mathbf{H}}$ be the maximum of $|\pi_0(\mathbf{Z}_{\mathbf{H},\mathrm{Norm}(y)})(k)|$ and $C_{\mathbf{H},\sigma}$ the maximum of $|H^1(\sigma,\pi_0(\mathbf{Z}_{\mathbf{H}, y \rtimes \sigma})(k))|$ for all $y \in I.$  Combining \eqref{firstinequality} and \eqref{secondinequality} gives

\begin{multline} \label{secondinequality2}
\frac{|\mathbf{Z}_{\mathbf{G},\delta \rtimes \sigma}(k)|}{|\mathbf{H}(k')|} \times \left| \mathbf{H}(k') \cap \{ \G(k')\text{-twisted conjugacy class of $\delta$}\} \right| \nonumber \\
\begin{split}
&\leq M_{\mathbf{H}} \times |\mathbf{Z}_{\mathbf{G},\delta \rtimes \sigma}(k)| \times \sum_{y \in I} \frac{1}{|\mathbf{Z}_{\mathrm{Norm}(y)}(k)|} \nonumber \\
&= M_{\mathbf{H}} \times |\mathbf{Z}_{\mathbf{G},\delta \rtimes \sigma}(k)| \times \sum_{y \in I} \frac{|\{ \mathbf{H}(k) \text{-conjugacy class of } \mathrm{Norm}(y)  \}|}{|\mathbf{H}(k)|} \nonumber \\
&\leq  M_{\mathbf{H}} \times \frac{|\mathbf{Z}_{\mathbf{G},\delta \rtimes \sigma}(k)|}{|\mathbf{Z}_{\mathbf{G},\mathrm{Norm}(\delta)}(k)|} \times C_{\mathbf{H},\sigma} \times \left( \frac{|\mathbf{Z}_{\mathbf{G},\mathrm{Norm}(\delta)}(k)|}{|\mathbf{H}(k)|} \times  \left| \mathbf{H}(k) \cap \{ \G(k)\text{-conjugacy class of $\delta$}\} \right| \right) \nonumber \\
&=  M_{\mathbf{H}} \times \frac{|\mathbf{Z}_{\mathbf{G},\delta \rtimes \sigma}(k)|}{|\mathbf{Z}_{\mathbf{G},\mathrm{Norm}(\delta)}(k)|} \times C_{\mathbf{H},\sigma} \times \tr \left( \mathrm{Norm}(\delta) | \mathrm{Ind}_{\mathbf{H}(k)}^{\mathbf{G}(k)} 1 \right).
\end{split}
\end{multline}

Because $\mathbf{Z}_{\G, \delta \rtimes \sigma}$ and $\mathbf{Z}_{\G, \mathrm{Norm}(\delta)}$ are inner forms \cite[$\S 1$ Lemma 1.1]{AC}, Lemma \ref{innerforms} gives
\begin{equation} \label{thirdinequality}
 \frac{|\mathbf{Z}_{\mathbf{G},\delta \rtimes \sigma}(k)|}{|\mathbf{Z}_{\mathbf{G},\mathrm{Norm}(\delta)}(k)|} \leq |\pi_0(\mathbf{Z}_{\mathbf{G},\delta \rtimes \sigma})(k)| \leq M_{\G,\sigma},
\end{equation}
where $M_{\mathbf{G},\sigma}$ is the maximum of $\pi_0(Z_{\G, \delta \rtimes \sigma})$ over all $\delta \in \G(k').$  Combining \eqref{traceofinduced}, \eqref{secondinequality}, and \eqref{thirdinequality} gives  
\begin{align*} 
c_{\Gamma_\mathfrak{p}}(\delta) = \tr \left( \delta \sigma | \mathrm{Ind}_{\mathbf{H}(O_E/\mathfrak{p}O_E)}^{\G(O_E/\mathfrak{p}O_E)} 1 \right) &\leq  M_{\mathbf{H}} \times M_{\G,\sigma} \times C_{\mathbf{H},\sigma} \times \tr \left( \mathrm{Norm}(\delta) | \mathrm{Ind}_{\mathbf{H}(k)}^{\mathbf{G}(k)} 1 \right) \\ 
&= M_{\mathbf{H}} \times M_{\G,\sigma} \times C_{\mathbf{H},\sigma} \times |\mathrm{Fix}\left(\mathrm{Norm}(\delta) | \mathbf{H}(k) \backslash \G(k) \right) |.
\end{align*} 

In Lemma \ref{centralizercomponents}, we show that the number of geometric components of the centralizer of every element of $\mathbf{H}(\overline{k})$ or $\mathbf{G}(\overline{k})$ is uniformly bounded over all $\mathfrak{p}$; this proves the same for twisted centralizers too, since the twisted centralizer of $\delta$ is an (inner) form of the centralizer of $\mathrm{Norm}(\delta).$  Thus, $M_{\mathbf{H}}, M_{\G,\sigma},$ and $C_{\mathbf{H},\sigma}$ are bounded by some constant $M = M(\G, \mathbf{H})$ depending only on $\G$ and $\mathbf{H}.$  We are reduced to bounding the right side of 
$$\frac{c_{\Gamma_\mathfrak{p}}}{[\Gamma^{\sigma}:\Gamma_\mathfrak{p}^{\sigma}]} \leq M \times \frac{ |\mathrm{Fix}\left(\mathrm{Norm}(\delta) | \mathbf{H}(k) \backslash \G(k) \right) |}{|\mathbf{H}(k) \backslash \G(k) |}.$$
\begin{itemize}
\item
Evidently,
$$ M \times \frac{ |\mathrm{Fix}\left(\mathrm{Norm}(\delta) | \mathbf{H}(k) \backslash \G(k) \right) |}{|\mathbf{H}(k) \backslash \G(k) |} \leq M,$$
uniformly for all $\delta$ and all $p.$

\item
To ease notation, let $\gamma = \mathrm{Norm}(\delta).$  By Lang's theorem,
$$ \frac{ |\mathrm{Fix}\left(\gamma | \mathbf{H}(k) \backslash \G(k) \right) |}{|\mathbf{H}(k) \backslash \G(k) |} =  \frac{ |\mathrm{Fix}\left(\gamma | (\mathbf{H} \backslash \G) \right)(k) |}{|(\mathbf{H} \backslash \G)(k) |}.$$
If $\gamma$ acts trivially on $\mathbf{H} \backslash \G,$ then $\gamma \in \bigcap_{g \in \G(\overline{F})} g \mathbf{H}_{\overline{F}}g^{-1},$ a normal subgroup of $\G_{\overline{F}}$ contained in $\mathbf{H}_{\overline{F}},$ implying that $\gamma = 1$ by hypothesis.  Therefore, $\mathrm{Fix}(\gamma | \mathbf{H} \backslash \G)$ is a proper subvariety of $\mathbf{H} \backslash \G$ for every $\gamma \neq 1.$  Since $\mathbf{H} \backslash \G$ is irreducible, $\mathrm{Fix}(\gamma | \mathbf{H} \backslash \G)$ must have strictly positive codimension in $\mathbf{H} \backslash \G.$  It follows that
\begin{equation*}
\frac{c_{\Gamma_\mathfrak{p}}(\delta)}{[\Gamma^{\sigma}:\Gamma_\mathfrak{p}^{\sigma}]} \leq M \times \frac{|\mathrm{Fix}(\mathrm{Norm}(\delta)| \mathbf{H} \backslash \G)(k)|}{(\mathbf{H} \backslash \G)(k)} \xrightarrow{p \rightarrow \infty} 0.  
\end{equation*}

\end{itemize}  
\end{proof} 

\begin{lem} \label{centralizercomponents}
Let $\G$ be an affine algebraic group over $O_{F,S}.$   The number of components of $\mathbf{Z}_{\G,x},$ where $x$ ranges over all elements of $\G(O_F / \mathfrak{p})$ for all primes ideals $\mathfrak{p} \notin S$ is uniformly bounded.
\end{lem}

\begin{proof}
$\G$ is a closed subvariety of $\SL_n \subset \mathrm{End}_n$ for some fixed $n.$  Suppose $\G$ is the simultaneous vanishing locus of polynomials $f_1,\ldots,f_m$ on the vector space $\mathrm{End}_n.$  Let $f_i$ have degree $d_i,$ the maximum degree of all the monomials in its support.  Fix a prime ideal $\mathfrak{p} \notin S$ and $x \in \G(O_F / \mathfrak{p}).$  Note that
$$\mathbf{Z}_{\G,x} = \mathbf{Z}_{\mathrm{End}_n,x} \cap V(f_1,\ldots,f_m).$$
Clearly, $\mathbf{Z}_{\G,x}$ is Zariski open and dense in its projective completion
$$\widetilde{\mathbf{Z}_{\G,x}} = \mathbb{P}(\mathbf{Z}_{\mathrm{End}_n,x}) \cap V(\widetilde{f}_1,\ldots,\widetilde{f}_m)$$ 
obtained by adding a hyperplane at infinity to $\mathrm{End}_n.$  Therefore, the number of connected components of $\widetilde{\mathbf{Z}_{\G,x}}$ equals the number of connected components of $\mathbf{Z}_{\G,x}.$  Thus,
\begin{align*}
|\pi_0(\mathbf{Z}_{\G,x})| &= | \pi_0(\widetilde{\mathbf{Z}_{\G,x}}) | \\
&\leq \text{degree of } \mathbb{P}(\mathbf{Z}_{\mathrm{End}_n,x}) \cap V(\widetilde{f}_1) \cap \cdots \cap V(\widetilde{f}_m) \\
&\leq 1 \cdot d_1 \cdots d_m, 
\end{align*}
where the final inequality follows by B\'{e}zout's theorem after noting that the degree of every $f_i$ can only decrease mod $\mathfrak{p}.$  This upper bound is independent of $\mathfrak{p}$ and $x.$
\end{proof}

\medskip
\noindent
{\it Remark.} The proof of Proposition \ref{boundinggrowth} is an adaptation of Shintani's arguments \cite{Shintani} proving the existence of a ``base change transfer"
$$\text{representations of } \GL_n(k) \leadsto \text{Galois-fixed representations of } \GL_n(k').$$ 
On the one hand, he sidesteps all component and endoscopy issues by working with $\G = \GL_n, \mathbf{H} =$ parabolic subgroup.  On the other hand, he proves an exact trace identity between matching principal series representations, an analogue of the fundamental lemma for $\GL_n(k).$  
 
\medskip

In the next sections we compute the $L^2$-Lefschetz numbers and twisted $L^2$-torsion and in particular prove Lemma \ref{Lkt}. 
We distinguish two cases: we first deal with the case where $\mathbb{E} = \mathbb{R}^p$ (the product case) and then deal with case 
where $\mathbb{E} = \mathbb{C}$. The general case easily reduces to these two cases.

\section{Computations on a product} \label{product}
Here we suppose that $\mathbb{E} = \mathbb{R}^p$. Then $G$ is the $p$-fold product of $\G (\mathbb{R})$ and $\sigma$ cyclically permutes the factors of $G$.
We will abusively denote by $G^{\sigma}$ the group $\G (\mathbb{R})$. Let $(\rho_0 , F_0)$ be an irreducible complex linear representation of 
$\G (\mathbb{C})$. We denote by $(\widetilde{\rho} , F)$ the corresponding complex finite dimensional $\sigma$-stable irreducible representation 
of $\widetilde{G}$. Recall that $F = F_0^{\otimes p}$, that $G$ acts by the tensor product representation $\rho_0^{\otimes p}$ and that 
$\sigma$ acts by the cyclic permutation $A: x_1 \otimes \ldots \otimes x_p \mapsto x_p \otimes x_1 \otimes \ldots \otimes x_{p-1}$. 
We finally let $X$ and $X^{\sigma}$ be the symmetric spaces corresponding to $G$ and $G^{\sigma}$ respectively, so $X = (X^{\sigma})^p$.

\subsection{Heat kernels of a product}
The heat kernels $H_t^{\rho, j}$ decompose as
\begin{equation} \label{productheatkernel}
H_t^{\rho , j}(g_1,\cdots,g_p) = \sum_{a_1 + \ldots + a_p = j} H_t^{\rho_0 , a_1}(g_1) \otimes \cdots \otimes H_t^{\rho_0 , a_p}(g_p).
\end{equation}
Now the twisted orbital integral of $H_t^{\rho , j}$ associated to the class of the identity element is given by
$$\left(\int_{G^{\sigma} \backslash G} H_t^{\rho , j} (g^{- \sigma} g) dg \right) \circ A_{\sigma}.$$
But because $H_t^{\rho , j} (g^{- \sigma} g)$ preserves all of the summands in the decomposition of \eqref{productheatkernel} and $\sigma$ maps the $(a_1, \ldots ,a_p)$-summand to the $(a_p,a_1,\ldots ,a_{p-1})$-summand, only those summands for which $j = pa$ and $a_1 = \ldots = a_p = a$ can contribute to the trace of the above twisted orbital integral.  Furthermore, by a computation identical to that done for scalar-valued functions in \cite[\S 8]{Lan}, we see that
\begin{equation*}
\left( \int_{G^{\sigma} \backslash G}  H_t^{\rho_0 ,a}(g_p^{-1}g_1) \otimes \ldots \otimes H_t^{\rho_0 , a}(g_{p-1}^{-1}g_p) dg \right) \circ A_\sigma = (-1)^{a^2(p-1)} H_t^{\rho_0 ,a} * \ldots * H_t^{\rho_0 ,a} (e).  
\end{equation*}
This implies that 
\begin{eqnarray*} 
\int_{G^{\sigma} \backslash G} h_t^{\rho , pa}(x^{-1} \sigma x) d\dot{x} &=& \tr\left[ \left( \int_{G^{\sigma} \backslash G}  H_t^{\rho_0 , a}(g_p^{-1}g_1) \otimes \ldots \otimes H_t^{\rho_0 , a}(g_{p-1}^{-1}g_p) d\dot{g} \right) \circ A_\sigma \right] \\
&=& (-1)^{a^2(p-1)} \tr \left(  H_t^{\rho_0 ,a} * \ldots * H_t^{\rho_0 ,a} (e)   \right) \\
&=&  (-1)^{a^2(p-1)} \tr \left(  H_{pt}^{\rho_0 ,a} (e)   \right) \\
&=& (-1)^{a^2(p-1)} h_{pt}^{\rho_0 ,a} (e). 
\end{eqnarray*}
Here $H_{pt}^{\rho_0 ,a}$ is an {\it untwisted} heat kernel on $X^{\sigma}$. Lemma \ref{Lkt} therefore follows from standard estimates (see e.g. \cite{BV}). Moreover,
computations of the $L^2$-Lefschetz number and of the twisted $L^2$-torsion immediately follow from the above explicit computation. 

\begin{thm}[$L^2$-Lefschetz number of a product] \label{T:62}
We have:
$$\mathrm{Lef}^{(2)} (\sigma , X , F) = \left\{ \begin{array}{ll}
(-1)^{\frac12 \dim X^{\sigma}} (\dim F_0 ) \frac{\chi (X_u^{\sigma})}{\vol (X_u^{\sigma})} & \mbox{ if } \delta (G^{\sigma}) = 0, \\
0 & \mbox{ if not}.
\end{array} \right.$$
Here $X_u^{\sigma}$ is the compact dual of $X^{\sigma}$ whose metric is normalized such that multiplication by $i$ becomes an isometry 
$T_{eK^{\sigma}} (X^{\sigma}) \cong \p \to i \p \cong T_{eK^{\sigma}} (X_u^{\sigma})$.
\end{thm}
\begin{proof}
First note\footnote{Beware that $\rho$ is not assumed to be strongly acyclic here !} that 
\begin{eqnarray*}
\mathrm{Lef}^{(2)} (\sigma , X , F) & = & \lim_{t \to +\infty} \int_{G^{\sigma} \backslash G} k_{t}^{\rho}(x^{-1} \sigma x) d\dot{x} \\
& = & \lim_{t \to +\infty} \sum_a (-1)^{pa} \int_{G^{\sigma} \backslash G} h_{t}^{\rho, pa}(x^{-1} \sigma x) d\dot{x} \\
& = & \lim_{t \to +\infty} \sum_a (-1)^{a} h_{pt}^{\rho_0 ,a} (e).
\end{eqnarray*}
The computation then reduces to the untwisted case for which we refer to \cite{Olbricht}.
\end{proof}

The computation of the twisted $L^2$-torsion similarly reduces to the untwisted case:

\begin{thm}[Twisted $L^2$-torsion of a product] \label{twistedtorsionproduct}
We have:
$$t_{X}^{(2)\sigma}(\rho) = p \cdot t_{X^{\sigma}}^{(2)}(\rho).$$
\end{thm}

\section{Computations in the case $\mathbb{E} = \mathbb{C}$}

Throughout this section, $\mathbb{E} = \mathbb{C}$. Then $G= \G (\C)$ is the group of complex points, $\sigma : G \to G$ is the real involution given by 
the complex conjugation and $G^{\sigma} = \G (\R)$. Recall that we fix a choice of Cartan involution $\theta$ of $G$ that commutes with $\sigma$.

\subsection{Irreducible $\sigma$-stable tempered representations of $G$}
Choose $\theta$-stable representatives $\h_1^0 , \ldots , \h_s^0$ of the $\G (\R)$-conjugacy classes of Cartan subalgebras in the Lie algebra $\g^0$ of $\G (\R)$.
For each $j \in \{1 , \ldots , s\}$ we write $\h_j^0 = \t_j \oplus \a_j$ for the decomposition of $\h_j^0$ w.r.t. $\theta$, i.e. $\a_j$ is the split part of $\h_j^0$ and
$\t_j$ is the compact part of $\h_j^0$. We denote by $\h_j$ the complexification of $\h_j^0$; note that $\a_j \oplus i \t_j$ and $\t_j \oplus i \a_j$ are 
resp. the split and compact part of $\h_j$. 

We now fix some $j$. To ease notations we will omit the $j$ index. 
Choose a Borel subgroup $B$ of $G=\G (\C)$ containing the torus $H$ which corresponds to $\h_j$.  Let $A$ and $T$ be resp. the split and
compact tori corresponding to $\a \oplus i \t$ and $\t \oplus i \a$. Write $\mu$ for the differential of a character of $T$ and $\lambda$ for the differential of a character of $A$. Note that $\mu$ is $\sigma$-stable if and only if $\mu$ is zero on $i \a$. 

Associated to $(\mu , \lambda)$ is a representation 
$$\pi_{\mu, \lambda} = \mathrm{ind}_{B}^{G} (\mu \otimes \lambda \otimes 1).$$

\begin{prop}[Delorme \cite{Delorme}] \label{Delorme}
Every irreducible $\sigma$-stable tempered representation of $G$ is equivalent to some $\pi_{\mu , \lambda}$ as above (for some $j$) where 
$\mu$ is zero on $i \a_j$ and $\lambda$ is zero on $\t_j$ and has pure imaginary image.
\end{prop}
Note that if  $\lambda$ is zero on $\t_j$ we may think of $\lambda$ as a real linear form $\a \to \C$.

We denote by $I_{\mu , \lambda}$ the underlying $(\g  , K)$-module. It is $\sigma$-stable and Delorme \cite[\S 5.3]{Delorme} define a particular extension 
to a $(\g , \widetilde{K})$-module, but we won't follow his convention here (see Convention I below).

\subsection{Computations of the Lefschetz numbers} If $\G (\R)$ has no discrete series Delorme \cite[Proposition 7]{Delorme} proves that 
for any admissible $(\g , \widetilde{K})$-module and any finite dimensional representation $(\widetilde{\rho} , F)$ of $\widetilde{G}$, we have:
$$\mathrm{Lef}(\sigma , F , V) = 0.$$
Even if $\G (\R)$ has discrete series Delorme's proof --- see also \cite[Lemma 4.2.3]{RohlfsSpeh} --- shows that 
$$\mathrm{Lef} (\sigma , F , I_{\mu , \lambda}) = 0$$
unless $\h^0 = \t$ is a compact Cartan subalgebra (so that $i \t$ is the split part of $\h$). In the latter case $\lambda=0$ (recall that $I_{\mu , \lambda}$
is assumed to be $\sigma$-stable); we will simply denote by $I_\mu$ the $(\g , \widetilde{K})$-module $I_{\mu, 0}$. The following proposition --- due to 
Delorme \cite[Th. 2]{Delorme}\footnote{Note that Delorme considers $\sigma$-invariants rather than traces, this introduces a factor $1/2$.} --- computes the Lefschetz numbers in the remaining cases.
\begin{prop} \label{P:delorme}
We have:
$$\mathrm{Lef} (\sigma , F , I_\mu ) = \left\{ \begin{array}{ll}
\pm 2^{\dim \t} & \mbox{ if } w \mu = 2 (\nu + \rho)_{| \t} \quad (w \in W) \\
0   & \mbox{ otherwise}.
\end{array} \right.$$
Here $W$ is the Weyl group of $(\g , \h)$ and the sign depends on the chosen extension of $I_{\mu}$ to a $(\g , \widetilde{K})$-module.
\end{prop}

\subsubsection*{Convention I} In the following we will always assume that the extension of a $\sigma$-discrete $I_{\mu}$ to a $(\g , \widetilde{K})$-module is 
s.t. that the sign in Proposition \ref{P:delorme} is positive. (See \cite[\S 4.2.5]{RohlfsSpeh} for more details.)

\subsection{Computations of the twisted $(\g , K)$-torsion}
Consider an arbitrary irreducible $\sigma$-stable tempered representation of $G$ associated to some $j$ and some $(\mu, \lambda)$ as in Proposition \ref{Delorme}.
Let $\mathbf{P}$ be the parabolic subgroup of $\G$ whose Levi subgroup $\mathbf{M} = {}^0 \mathbf{M} \mathbf{A}_P$ is the centralizer in $\G$ of $\a$. 
We have $B \subset \mathbf{P} (\C)$ and we may write $\pi_{\mu , \lambda}$ as the induced representation
$$\pi_{\mu , \lambda} = \mathrm{ind}_{\mathbf{P} (\C)}^{\G (\C)} ( \pi_{\mu}^{{}^0 \mathbf{M} (\C)} \otimes \lambda),$$
where 
$$\pi_{\mu,0}^{{}^0 \mathbf{M} (\C)} = \mathrm{ind}_{B \cap {}^0 \mathbf{M} (\C)}^{{}^0 \mathbf{M} (\C)} (\mu_{|\t} \otimes 0)$$ 
is a tempered ($\sigma$-discrete) representation of ${}^0 \mathbf{M} (\C)$ and we think of $\lambda$ --- seen as real linear form $\a \to \C$ --- as (the differential of)
a character of $\mathbf{A}_P (\C)$.

\subsubsection*{Convention II} In the following we fix the extension of $I_{\mu , \lambda}$ to a $(\g  , \widetilde{K})$-module to be the one associated to the interwining
operator $A_G = \mathrm{ind} _{\mathbf{P} (\C)}^{\G (\C)} (A_M \otimes 1)$ where $A_M$ is chosen according to Convention I.

Let $K_M = K \cap {}^0 \mathbf{M} (\C)$. Since $\sigma$ stabilizes ${}^0 \mathbf{M} (\C)$, $\mu$, {\it etc}$\ldots$ it follows from Frobenius reciprocity and \eqref{Lef'} that we have:
\begin{equation} \label{Lef'2}   
\mathrm{Lef} {}' (\sigma , F , I_{\mu , \lambda}) =\mathrm{trace} \left( \sigma \; | \; \left[ \detp [1 - (\g / \k)^*] \otimes F \otimes \pi_{\mu,0}^{{}^0 \mathbf{M} (\C)} \right]^{K_M} \right). \end{equation}
Write 
$$\g / \k = {}^0\mathfrak{m} / \k_M \oplus \a  \oplus \mathfrak{n}.$$

It follows from \eqref{add} that --- as a $\widetilde{K}_M$-module --- we have:
$$\detp [1-(\g / \k)^* ] = \det[1-({}^0\mathfrak{m} / \k_M)^*]  \otimes \detp [1- \a^*\oplus \mathfrak{n}^*] \oplus \detp[1-({}^0\mathfrak{m}/ \k_M)^*]  \otimes \det [1- \a^*\oplus \mathfrak{n}^*],$$
with
\begin{equation} \label{tracezeropart}
\det [1- \a^*\oplus \mathfrak{n}^*]= \det [1- \a^*] \otimes \det [1- \mathfrak{n}^*]
\end{equation}
and 
\begin{equation*}
\detp [1- \a^*\oplus \mathfrak{n}^*] = \det[1- \mathfrak{n}^*] \otimes \detp [1- \a^*] \oplus \detp [1- \mathfrak{n}^*] \otimes \det [1- \a^*].
\end{equation*}
Now two simple lemmas:

\begin{lem}  \label{L:triv}
(1) We have $\det [1- \a^*] = 0$ as a virtual $\widetilde{K}_M$-module unless $\a^{\sigma} = 0$.

(2) We have $\detp [1- \a^*] = 0$ as a virtual $\widetilde{K}_M$-module unless $\dim \a^{\sigma} \leq 1$. 
\end{lem}

\begin{proof}
For any $\delta \in \widetilde{K}_M,$ 
$$\tr (\delta | \det[1-\a^*]) = \det(1-\delta | \a^*), \tr( \delta | \detp [1-\a^*]) = \frac{d}{dt}|_{t = 1} \det(t \cdot 1 - \delta | \a^*)$$
 cf. \S \ref{lefschetznumber}.  Write $\delta = \epsilon k,$ where $\epsilon \in \{ 1, \sigma \}$ and $k \in K_M.$  

For any $X \in \a^{*},$
$$\delta X = \epsilon k X = \epsilon X$$
since $K_M$ centralizes $\a^*.$  Thus, 
$$\dim \{ \text{+1-eigenspace of } \delta \} \geq \dim \; (\a^*)^{\sigma}.$$      
In particular, 
\begin{itemize}
\item[(1)]
if $\dim \a^{\sigma} > 0,$ then 
$$\det(1 - \delta | \a^*) = 0.$$

\item[(2)]
if $\dim \a^{\sigma} > 1,$ then $\det(t \cdot 1 - \delta | \a^*)$ vanishes to order at least 2 at $t = 1,$ whence
$$\frac{d}{dt}|_{t = 1}\det(t \cdot 1 - \delta | \a^*) = 0.$$
\end{itemize}
\end{proof}

\begin{lem} \label{virtuallytrivial}
Let $V$ be a finite dimensional $\widetilde{K}_M$-module and $\tau$ any admissible $\widetilde{K}_M$-module, i.e. a $\widetilde{K}_M$-module all of whose $K_M$-isotypic subspaces are finite dimensional.  Suppose $V$ is virtually trivial.  Then $[V \otimes \tau]^{K_M}$ is finite dimensional and
$$\tr(\sigma | [V \otimes \tau]^{K_M}) = 0.$$
\end{lem}

\begin{proof}
Finite dimensionality is immediate since $\tau$ is admissible.  Let $\zeta$ be a finite dimensional subrepresentation of $\tau$ such that $[V \otimes \tau]^{K_M} = [V \otimes \zeta]^{K_M}.$
   
Since $K_M$ is compact, taking $K_M$-invariants is an exact functor from the category of finite dimensional $K_M$-modules to the category of finite dimensional $\sigma$-modules.  Virtually trivial $\widetilde{K}_M$-modules therefore map to virtually trivial $\sigma$-modules.  Thus, $[V \otimes \zeta]^{K_M}$ is virtually trivial.  In particular,
$$\tr(\sigma | [V \otimes \tau]^{K_M}) = \tr(\sigma | [V \otimes \zeta]^{K_M}) = 0.$$
\end{proof}

In particular we conclude that \eqref{Lef'2} is zero unless $\dim \a \leq 1$. In the following we compute \eqref{Lef'2} in the two remaining cases.

\subsection{Computation of \eqref{Lef'2} when $\dim \a = 1$}
Assume that $\dim \a =1$.  It then follows from Lemmas \ref{L:triv} and \ref{virtuallytrivial} that

$$\tr (\sigma \; | \; \detp [1-(\g / \k)^* ]) = \tr (\sigma \; | \; \det[1-({}^0\mathfrak{m} / \k_M)^*]  \otimes \detp [1- \a^*\oplus \mathfrak{n}^*]).$$
We can therefore compute
\begin{equation} \label{Lef'3}
\begin{split}
\mathrm{Lef} {}' (\sigma , F , I_{\mu , \lambda}) & =  \tr \left( \sigma \; | \; \left[ \det[1-({}^0\mathfrak{m}/ \k_M)^*] \otimes \detp [1- \a^*\oplus \mathfrak{n}^*] \otimes F \otimes \pi_{\mu,0}^{{}^0 \mathbf{M} (\C)} \right]^{K_M} \right) \\
& = \mathrm{Lef} (\sigma , \detp [1- \a^* \oplus \mathfrak{n}^* ] \otimes F , I_{\mu}^{{}^0 \mathbf{M} (\C)}).
\end{split}
\end{equation}

For each $w \in W$ we let $\nu_w$ be the restriction of $w(\rho + \nu)$ to $\t$. 
Let $[W_{K_M} \backslash W]$ be the set of $w \in W_U$ such that $\nu_w$ is dominant as a weight on $\t$ (with respect to the roots of $\t$ on $\k_M$), i.e.:
$$[W_{K_M} \backslash W] = \{ w \in W \; : \; \langle \nu_w, \beta \rangle \geq 0 \mbox { for all } \beta \in \Delta^{+}(\mathfrak{t}, \mathfrak{k}_M) \}.$$
This is therefore a set of coset representatives for $W_{K_M}$ in $W$. 

\begin{prop} \label{P68}
Assume $\dim \a = 1$. Then we have:
$$\mathrm{Lef} {}' (\sigma , F , I_{\mu , \lambda}) = \left\{ \begin{array}{ll}
 \mathrm{sgn} (w) 2^{\dim \t} & \mbox{ if } \mu = 2\nu_w \mbox{ for some } w \in [W_{K_M} \backslash W] \\
0 & \mbox{ otherwise}.
\end{array} \right.$$
\end{prop}
\begin{proof} We shall apply Proposition \ref{P:delorme} to the twisted space associated to ${}^0 \mathbf{M} (\C)$ to compute \eqref{Lef'3}. 
To do so directly, we would need to decompose the virtual representation 
\begin{equation} \label{*}
\detp [1- \a^* \oplus \mathfrak{n}^*] \otimes F 
\end{equation}
into irreducibles. This can be done by hand 
by reducing to the cases where $\G$ is simple of type $\SO (p,p)$ with $p$ odd, or of type $\SL(3)$. Instead we note that Proposition \ref{P:delorme} implies that only the essential $\sigma$-stable subrepresentations of \eqref{*} contribute to the final expression in \eqref{Lef'3}.
We may therefore reduce to considering the virtual representation
$$(\detp [1- \a_0^* \oplus \n_0^*] \otimes (\detp [1-\a_0^* \oplus \n_0^*])^{\sigma}) \otimes (F_0 \otimes F_0^{\sigma}) = (\detp [1- \a_0^* \oplus \n_0^*] \otimes F_0) \otimes (\detp [1- \a_0^* \oplus \n_0^*] \otimes F_0)^{\sigma}.$$
Here we have realized the $\widetilde{K}_M$-representations $\a$ and $\n$ as representations in $\a_0 \otimes \a_0^{\sigma}$ and $\n_0 \otimes \n_0^{\sigma}$, resp.

Now it follows from Proposition \ref{P:delorme} that $\mathrm{Lef} (\sigma , F , I_{\mu , \lambda}) =0$ unless $\mu = 2\mu_0$ where $\mu_0-\rho$ is the highest weight of a finite dimensional representation of ${}^0 \mathbf{M} (\R)$. Next we use that if $\theta_0$ denotes the discrete series 
of ${}^0 \mathbf{M} (\R)$ with infinitesimal character $\mu_0$ and $H_0$ a finite dimensional representation of ${}^0 \mathbf{M} (\R)$ of highest weight $\nu$ 
then the (untwisted) Euler-Poincar\'e characteristic 
\begin{equation*}
\begin{split}
\chi (H_0 , \theta_0) & := \dim [\theta_0 \otimes \det[ 1- ({}^0\mathfrak{m} (\R)/ \k_{M, \R})^* ] \otimes H_0]^{K_{M}^{\sigma}} \\
& = \left\{ \begin{array}{ll}
(-1)^{\frac12 \dim ({}^0\mathfrak{m} (\R)/ \k_{M, \R})} & \mbox{ if } w \mu_0 = \nu + \rho \quad (w \in W_M) \\
0 & \mbox{ otherwise}.
\end{array} \right.
\end{split}
\end{equation*}
We can extend $\chi (\cdot , \theta_0)$ to any virtual representation. Applying the above to $H_0 = \detp[1-\a_0^* \oplus \n_0^*] \otimes F_0$, it follows 
from Proposition \ref{P:delorme} and \eqref{Lef'3} that we have:
\begin{equation*} 
\begin{split}
\mathrm{Lef} {}' (\sigma , F , I_{\mu , \lambda}) & = 2^{\dim \t } \times \# \{ \text{irreducible } \sigma \text{-stable } {}^0 \mathbf{M}(\C) \text{ subrepresentations of } \detp [1-\mathfrak{a}^* \oplus \mathfrak{n}^*] \otimes F \\
& \quad \text{ with same infinitesimal character as } I_{2\mu_0,0} \} \\
&=  2^{\dim \t } \times \# \{ \text{irreducible } \sigma \text{-stable } {}^0 \mathbf{M}(\C) \text{ subrepresentations of } \\
& \quad (\detp [1- \mathfrak{a}_0^* \oplus \mathfrak{n}_0^*] \otimes F_0) \otimes  (\mathrm{det}'[1- \mathfrak{a}_0^* \oplus \mathfrak{n}_0^*] \otimes F_0)^{\sigma} \\
& \quad  \text{ with same infinitesimal character as } I_{2\mu_0,0} \} \\
&=  2^{\dim \t } \times \# \{ \text{irreducible } {}^0 \mathbf{M}(\R) \text{ subrepresentations of } \detp [1- \mathfrak{a}_0^* \oplus \mathfrak{n}_0^* ] \otimes F_0 \\
& \quad \text{ with infinitesimal character equal to that of } \theta_0 \} \\
&= 2^{\dim \t} (-1)^{\frac12 \dim ({}^0\mathfrak{m} (\R)/ \k_{M, \R})} \times \chi( \detp [1- \mathfrak{a}_0^* \oplus \mathfrak{n}_0^* ] \otimes F_0, \theta_0) \\
& = 2^{\dim \t } (-1)^{\frac12 \dim ({}^0\mathfrak{m} (\R)/ \k_{M, \R})} \dim [\theta_0 \otimes \det[ 1- ({}^0\mathfrak{m} (\R)/ \k_{M, \R}) ] \otimes \detp[1-\a_{0}^* \oplus \n_{0}^*] \otimes F_0]^{K_{M}^{\sigma}} \\
& = 2^{\dim \t } (-1)^{\frac12 \dim ({}^0\mathfrak{m} (\R)/ \k_{M, \R})+1} \dim [I_{\mu_0} \otimes \detp [1 - (\g (\R)/ \k_{\R})^*] \otimes F_0]^{K^{\sigma}} \\
& =  2^{\dim \t}  (-1)^{\frac12 \dim ({}^0\mathfrak{m} (\R)/ \k_{M, \R})} \times \mathrm{det}'(F_0, I_{\mu_0}).
\end{split}
\end{equation*}
We are therefore reduced to the untwisted case. And the proposition finally follows from the computations made in \cite[\S 5.6]{BV}.
\end{proof}

\medskip
\noindent
{\it Remark.} 
1. The proof above and the transfer of infinitesimal characters under base change shows that 
$$(\text{twisted heat kernel for } F)(2t) \xrightarrow{\text{transfer}} 2^{\dim \t}  (-1)^{\frac12 \dim ({}^0\mathfrak{m} (\R)/ \k_{M, \R})} \times ( \text{heat kernel for } F_0)(t)$$
for base change $\C / \R$.

2. Note that this base change calculation includes, as a special case, that of a product $\G = \G' \times \G'$.  But we've worked out separately that 
$$\mathrm{Lef}'(\sigma,F_0 \otimes F_0, \pi \otimes \pi) = 2 \mathrm{det}'(F_0, \pi).$$  
There is no contradiction here: if either side is non zero then $\dim \a = 1,$
but in that special case $\dim \a = \dim \t$ since the real group is in fact a complex group. 

\medskip

\subsection{Computation of \eqref{Lef'2} when $\dim \a = 0$} We now assume that $\dim \a =0$ and follow an observation made by Mueller and Pfaff \cite{MuellerPfaff}. 
In that case $\mathbf{M} = \G$,  $K_M = K$ and $\pi_{\mu , \lambda} = \pi_{\mu , 0}$ is $\sigma$-discrete. 
Now $\dim \g (\C) / \k$ equals $2d,$ where $d$ is the dimension of the symmetric space associated to $\G(\R)$. 
Note that --- as $\widetilde{K}$-modules --- we have 
$$\wedge^i (\g  / \k)^* \cong \wedge^{2d-i} (\g  / \k)^*, \quad i=0, \ldots , 2d.$$
It follows that as $\widetilde{K}$-representations we have:
$$\detp[1 - (\g  / \k)^*] = d \det [1-(\g / \k)^*].$$
This implies that 
$$\mathrm{Lef} {}' (\sigma , F , I_{\mu , 0}) = d \ \mathrm{Lef} (\sigma , F , I_{\mu , 0}).$$
Proposition \ref{P:delorme} therefore implies:

\begin{prop} \label{P610}
Let $\pi_{\mu}$ be a $\sigma$-discrete representation of $G$. Then we have:
$$\mathrm{Lef} {}' (\sigma , F , I_{\mu}) = \left\{ \begin{array}{ll}
2^{\dim \t} \dim (\g^0 / \k^0 ) & \mbox{ if } w \mu = 2 (\nu + \rho)_{| \t} \quad (w \in W) \\
0   & \mbox{ otherwise}.
\end{array} \right. .$$
\end{prop}

\subsection{Proof of Lemma \ref{Lkt}}
If $\phi$ is any smooth compactly supported function on $G$, Bouaziz \cite{Bouaziz} shows that 
\begin{equation} \label{eq:bouaziz}
\int_{G^{\sigma} \backslash G} 
\phi (x^{-1} \sigma x) d\dot{x} = \phi^G (e)
\end{equation}
where $\phi^G \in C_c^{\infty} (\G (\R))$ is the transfer of $\phi$. Now we can use the Plancherel theorem of Herb and Wolf \cite{HerbWolf}
(as in \cite[Proposition 4.2.14]{RohlfsSpeh}) and get
$$\phi^G (e) = \sum_{\pi \ {\rm discrete}} d(\pi) \tr \ \pi (\phi^G) + \int_{\rm tempered} \tr \ \pi (\phi^G)  d\pi.$$
We can group the terms into stable terms since all terms in a $L$-packet have the same Plancherel measure \cite{HC}. We write $\pi_\varphi$ for the sum of the representations in an $L$-packet
$\varphi$ and $d\pi_\varphi = d\pi$ for the corresponding measure. We then obtain
$$\phi^G (e) = \sum_{\substack{{\rm elliptic} \\ L-{\rm packets} \ \varphi}} d(\varphi) \tr \ \pi_\varphi (\phi^G) + \int_{\substack{{\rm non \ elliptic} \\ {\rm bounded} \ L-{\rm packets} \ \varphi}} \tr \ \pi_\varphi (\phi^G)  d\pi_\varphi.$$
Now we use transfer again.  Indeed, Clozel \cite{Clozel} shows that if $\varphi$ is a tempered $L$-packet and $\widetilde{\pi}_\varphi$ the sum of the twisted representations of $\widetilde{G}$ associated to $\varphi$ by base-change, we have:
$$\tr \ \pi_{\varphi} (\phi^G) = \tr \ \widetilde{\pi}_\varphi (\phi).$$
We conclude:
\begin{equation} \label{eq:bouaziz2}
\int_{G^{\sigma} \backslash G} 
\phi (x^{-1} \sigma x) d\dot{x} = \sum_{\substack{{\rm elliptic} \\ L-{\rm packets} \ \varphi}} d(\varphi) \tr \ \widetilde{\pi}_\varphi (\phi) + \int_{\substack{{\rm non \ elliptic} \\ {\rm bounded} \ L-{\rm packets} \ \varphi}} \tr \ \widetilde{\pi}_\varphi (\phi)  d\pi_\varphi.
\end{equation}
We want to apply this to the function $\phi = k_t^{\rho , \sigma}$. Since it is not compactly supported we have to explain why \eqref{eq:bouaziz2} still holds for functions $\phi$ in the Harish-Chandra Schwartz space. We first note we have already checked (in \S 4.8) that the distribution 
$$\phi \mapsto \int_{G^{\sigma} \backslash G} 
\phi (x^{-1} \sigma x) d\dot{x}$$
extends continuously to the Harish-Chandra Schwartz space, i.e. it defines a \emph{tempered} distribution. Now for $\phi$ compactly supported Bouaziz \cite[Th\'eor\`eme 4.3]{Bouaziz} proves that we have (recall that we suppose that $H^1 (\sigma , G) = \{1 \}$):
\begin{equation} \label{eq:inversion}
\int_{G^{\sigma} \backslash G} 
\phi (x^{-1} \sigma x) d\dot{x} = \int_{(\mathfrak{g}_a^* / G)^\sigma} \left( \frac12 \sum_{\tau \in X(f)} Q_\sigma (f , \tau) \tr \ \Pi_{f , \tau} (\phi ) \right) dm (G \cdot f).
\end{equation}
We refer to \cite{Bouaziz} for all undefined notations and simply note that 
\begin{itemize}
\item the $\Pi_{f , \tau}$ are tempered (twisted) representation, and
\item if $\phi$ belongs to Harish-Chandra Schwartz space, the (twisted) characters $\Theta_{f,\tau} (\phi) = \tr \ \Pi_{f , \tau} (\phi) $ define rapidly decreasing functions of $f$. 
\end{itemize}
Bouaziz does not explicitly compute the function 
$Q_\sigma (f, \tau)$ but proves however that it grows at most polynomially in $f$. It therefore follows that the distribution defined by the right hand side of \eqref{eq:inversion} also extends continuously to the Harish-Chandra Schwartz space.\footnote{In fact, we need that (twisted) tempered characters are rapidly decreasing in the parameters ``Schwartz-uniformly'' in $\phi$.  But this holds because of known properties of discrete series characters combined with the fact that the constant term operator $\phi \mapsto \phi^{(P)}$ are all Schwartz-continuous.} We conclude that \eqref{eq:inversion} still holds when $\phi$ belongs to Harish-Chandra Schwartz space.

We may now group the characters $\Theta_{f, \tau}$ into finite packets to get (all) stable tempered characters, as in \cite[\S 7.3]{Bouaziz} and denoted $\widetilde{\Theta}_{\lambda}$ there. Then the right hand side of \eqref{eq:inversion} becomes 
$$\sum_{\mathfrak{a} \in \mathrm{Car}(\mathfrak{g}^0)/G^0 } 2^{-\frac12 (\dim G^\sigma + \mathrm{rank} G^{\sigma})} |W (G , \mathfrak{a})|^{-1} \int_{\mathfrak{a}^*} p_{\sigma, \sigma} (\lambda) \Pi_{\mathfrak{g}^0}^\sigma (\lambda) \widetilde{\Theta}_\lambda (\phi) d \eta_{\mathfrak{a}} (\lambda).$$
Here again we refer to \cite{Bouaziz} for notations. Then \cite[Eq. (3), (4), (5) and (6) p. 287]{Bouaziz} imply that the right hand side of \eqref{eq:inversion} is equal to 
$$\sum_{\mathfrak{a} \in \mathrm{Car}(\mathfrak{g}^0)/G^0 } |W (G , \mathfrak{a})|^{-1} \int_{\mathfrak{a}^*} p_{1, 1} (\lambda) \Pi_{\mathfrak{g}^0 / \mathfrak{a}} (\lambda) \widetilde{\Theta}_{2\lambda} (\phi) d \eta_{\mathfrak{a}} (\lambda).$$
Here the stable character $\widetilde{\Theta}_{2\lambda}$ is the transfer of the character of a tempered packet, see \cite[7.3(a)]{Bouaziz}) that is parametrized by $\lambda$ there.
Finally in this parametrization the measure against which we integrate $\widetilde{\Theta}_{2\lambda} (\phi)$ can be identified with the Plancherel measure (see the very beginning of the proof of \cite[Th\'eor\`eme 7.4]{Bouaziz}) and we conclude that \eqref{eq:bouaziz2} extends to Harish-Chandra Schwartz space. It applies in particular to $\phi= k_t^{\rho , \sigma}$, and using Lemma \ref{L1} and Propositions \ref{P68} and \ref{P610}, we conclude that
\begin{equation} \label{maineq}
\int_{G^{\sigma} \backslash G} 
k_t^{\rho , \sigma} (x^{-1} \sigma x) d\dot{x} = 2^{\dim \mathfrak{t}} \int_{{\rm tempered}} e^{-t (\Lambda_\rho - \Lambda_\pi)} \detp (F_0 , \pi) d\pi.
\end{equation}
Lemma \ref{Lkt} therefore follows from the untwisted case for which we refer to \cite{BV}.

\medskip

We furthermore deduce from \eqref{maineq} (and the computation in the untwisted case) the following theorem.

\begin{thm} \label{torsioncomparison}
We have :
$$t_X^{(2) \sigma} (\widetilde{\rho})  = 2^{\dim \t} t_{X^{\sigma}}^{(2)} (\rho).$$
\end{thm}

Similarly we have:

\begin{thm} \label{lefschetzcomparison}
We have:
$$\mathrm{Lef}_X^{(2)\sigma}(\widetilde{\rho})  = 2^{\dim \mathfrak{t}} \cdot \chi_{X^{\sigma}}^{(2)}(\rho).$$
\end{thm}

\medskip
\noindent
{\it Remark.} 
It follows from Theorems \ref{T:62} and \ref{lefschetzcomparison} that if $\delta (G^{\sigma})=0$ then $\mathrm{Lef}_X^{(2)\sigma}(\widetilde{\rho})$ is non zero. Proposition \ref{prop:intro} of the Introduction therefore follows from the limit formula proved in Corollary \ref{C:47}. 
\medskip

\section{General base change}

\subsection{Twisted torsion of product automorphisms}

Let $\G /\R$ be semisimple, $\sigma$ be an automorphism of $\G,$ and $\rho$ a $\sigma$-equivariant representation of $\G.$ 
Let $\G' /\R$ be semisimple, $\sigma'$ be an automorphism of $\G',$ and $\rho'$ a $\sigma'$-equivariant representation of $\G'.$

\begin{lem} \label{productaut}
There is an equality
$$t_{X_{G \times G'}}^{(2) \sigma \times \sigma'}(\rho \boxtimes \rho') = t_{X_G}^{(2) \sigma}(\rho) \cdot \mathrm{Lef}_{X_{G'}}^{(2) \sigma'}(\rho') +  \mathrm{Lef}_{X_G}^{(2) \sigma}(\rho) \cdot  t_{X_{G'}}^{(2) \sigma'}(\rho')$$
\end{lem}
\begin{proof}
Let $M,M'$ be compact Riemannian manifolds together with equivariant metrized local systems $L \rightarrow M$ and $L' \rightarrow M'.$  L\"{u}ck \cite[Proposition 1.32]{luck} proves that

$$t^{\sigma \times \sigma'}(M \times M';L \boxtimes L') = t^{\sigma}(M,L) \cdot \mathrm{Lef}(\sigma',M',L') + \mathrm{Lef}(\sigma,M,L) \cdot t^{\sigma'}(M',L').$$

Furthermore, Theorem \ref{twistedtorsionlimitmultiplicity} shows that

$$t^{(2) \sigma \times \sigma}_{X_{G \times G'}}(\rho \boxtimes \rho') = \lim_{n \rightarrow \infty} \frac{\log \tau^{\sigma \times \sigma'}(\Upsilon_n \backslash X_{G \times G'} )}{\vol(\Upsilon_n \backslash G \times G')}$$

for any sequence of subgroups $\Upsilon_n$ satisfying the hypotheses therein.  The sequence $\Upsilon_n = \Gamma_n \times \Gamma_n',$ where $\Gamma_n$ (resp. $\Gamma'_n$) is a chain of $\sigma$-stable (resp. $\sigma'$-stable) normal subgroups of $G$ (resp. $G'$) with trivial intersection, satisfies the hypotheses of Proposition \ref{WLF} and Theorem \ref{twistedtorsionlimitmultiplicity}.  Therefore,

\begin{eqnarray*}
t^{(2) \sigma \times \sigma'}_{X_{G \times G'}}(\rho \boxtimes \rho') &=& \lim_{n \rightarrow \infty} \frac{t^{\sigma}(\Gamma_n \backslash X_G, \rho)}{\vol(\Gamma_n^{\sigma} \backslash G^{\sigma})} \cdot \frac{\mathrm{Lef}(\sigma',\Gamma'_n \backslash X_{G'} ,\rho')}{\vol(\Gamma^{'\sigma'}_n \backslash G^{'\sigma'})} + \frac{\mathrm{Lef}(\sigma,\Gamma_n \backslash X_G,\rho)}{\vol(\Gamma_n^{\sigma} \backslash G^{\sigma})} \cdot \frac{t^{\sigma'}(\Gamma_n \backslash X_{G'},\rho')}{\vol(\Gamma_n^{'\sigma'} \backslash G^{'\sigma'})} \\
&=&  t^{(2) \sigma}_{X_G}(\rho) \cdot \mathrm{Lef}_{X_{G'}}^{(2) \sigma'}(\rho') +  \mathrm{Lef}_{X_G}^{(2) \sigma}(\rho) \cdot t^{(2) \sigma'}_{X_{G'}}(\rho').
\end{eqnarray*}
\end{proof}

\subsection{} Now let $\mathbb{E}$ be an \'{e}tale $\R$-algebra; concretely, $\mathbb{E} = \mathbb{R}^r \times \mathbb{C}^s.$ 
Fix $\sigma \in \Aut(\mathbb{E} / \R)$. The automorphism $\sigma$ permutes the factors of $\mathbb{E}$ and so induces a decomposition of the factors of $\mathbb{E}$ into its set of orbits $\mathcal{O}: \mathbb{E} = \prod_{o \in \mathcal{O}} \mathbb{E}_o.$  Each orbit is either

\begin{itemize}
\item[(a)]
a product of real places acted on by cyclic permutation,

\item[(b)]
a product of complex places acted on by cyclic permutation, or

\item[(c)]
a single complex place acted on by complex conjugation.
\end{itemize}      

Let $\G$ be a semisimple group over $\R.$  Let $\rho$ be a representation of $\G / \R$ and $\widetilde{\rho}_o$ the corresponding representation of $\mathrm{Res}_{\mathbb{E}_o / \R}.$  In particular, $\widetilde{\rho} = \rho \otimes \bar{\rho}$ is the corresponding representation of $\mathrm{Res}_{\C / \R}\G.$ The automorphism $\sigma$ induces a corresponding automorphism of the group $\mathrm{Res}_{\mathbb{E} / \R} \G.$  There is a decomposition 
$$\mathrm{Res}_{\mathbb{E} / \R} \G = \prod_{o \in \mathcal{O}} \mathrm{Res}_{\mathbb{E}_o / \R} \G$$
with respect to which $\sigma$ acts as a product automorphism.  

\begin{itemize}
\item Theorem \ref{twistedtorsionproduct} shows that 
$$t^{(2) \sigma}_{X_{\G(\mathbb{E}_o)}}(\widetilde{\rho}_o) =|o| \cdot t^{(2)}_{X_{\G(\R)}}(\rho)$$
and 
$$\mathrm{Lef}^{(2) \sigma}_{X_{\G(\mathbb{E}_o)}}(\widetilde{\rho}_o) = \chi^{(2)}_{X_{\G(\R)}}(\rho) $$
for the orbits of type (a).
\item Theorem \ref{twistedtorsionproduct} shows that 
$$t^{(2) \sigma}_{X_{\G(\mathbb{E}_o)}}(\widetilde{\rho}_o) =|o| \cdot t^{(2)}_{X_{\G(\C)}}(\widetilde{\rho})$$
and
$$\mathrm{Lef}^{(2)\sigma}_{X_{\G(\mathbb{E}_o)}}(\widetilde{\rho}_o) = \chi^{(2)}_{X_{\G(\C)}}(\widetilde{\rho})$$
for the orbits of type (b).
 \item Theorem \ref{torsioncomparison} proves that
$$t^{(2) \sigma}_{X_{\G(\mathbb{E}_o)}}(\widetilde{\rho}_o) = 2^{\dim \mathfrak{t}} t_{X_{\G(\R)}}^{(2)}(\rho)$$
and
$$\mathrm{Lef}^{(2)\sigma}_{X_{\G(\mathbb{E}_o)}}(\widetilde{\rho}_o) = \chi^{(2)}_{X_{\G(\R)}}(\rho)$$
for the orbits of type (c).
\end{itemize} 

The aggregate of these three examples, together with Theorem \ref{productaut}, allows us to compute the twisted $L^2$-torsion for arbitrary base change.

\begin{thm} We have
$$t^{(2)\sigma}_{X_{\G(\mathbb{E})}}(\widetilde{\rho}_{\mathbb{E}}) \neq 0$$ 
if and only if $\delta( \G(\mathbb{E})^{\sigma}) = 1.$  
\end{thm}
\begin{proof}
Suppose that there are $n$ orbits. Using Lemma \ref{productaut}, we expand $\tau^{(2)\sigma}_{X_{\G(\mathbb{E})}}(\widetilde{\rho}_{\mathbb{E}})$ as a sum of $n$ terms.  Each summand is a product of $\mathrm{Lef}^{(2) \sigma}_{X_{\G(\mathbb{E}_o)}}(\widetilde{\rho}_o)$ for $n - 1$ of the orbits $o$ and of $t^{(2)\sigma}_{X_{\G(\mathbb{E}_o)}}(\widetilde{\rho}_o)$ for the remaining orbit $o.$  Thus, exactly $n-1$ of the $\mathrm{Lef}^{(2) \sigma}$'s must be non-zero and the remaining $t^{(2) \sigma}$ must be non-zero; say $t^{(2) \sigma}_{X_{\G(\mathbb{E}_{o^{*}})}}(\rho_{o^{*}}) \neq 0.$  But by the preceding computations relating $t^{(2) \sigma}$ and $\mathrm{Lef}^{(2) \sigma}$ to their untwisted analogues, this is possible if and only if $\delta( \G(\mathbb{E}_{o^{*}})^{\sigma}) = 1$ and $\delta( \G(\mathbb{E}_o)^{\sigma}) = 0$ for all $o \neq o^{*}.$  This is equivalent to $\delta(\G(\mathbb{E})^{\sigma}) = 1.$
\end{proof}

\section{Application to torsion in cohomology}
\subsection{Generalities on Reidemeister torsion and the Cheeger-M\"{u}ller theorem}

Let $A^{\bullet}$ be a finite chain complex situated in degree $\geq0$ of $K$-vector spaces for a field $K.$  Suppose the chain groups $A^i$ and the cohomology groups $H^i(A^{\bullet})$ are equipped with volume forms, i.e. with non-zero elements $\omega_i \in \mathrm{det}(A^i)^{*}$ and $\mu_i \in (\mathrm{det}(H^i(A^{\bullet}) )^{*}.$ These define elements $\omega$ in 
$\mathrm{det}(A^{\bullet})^*$ and $\mu$ in $\mathrm{det}(H(A^{\bullet}))^*$, where
$$\mathrm{det}(A^{\bullet}) := \mathrm{det}(A_0) \otimes \mathrm{det}(A_1)^{-1} \otimes \mathrm{det}(A_2) \otimes \ldots $$
and similarly for $\mathrm{det}(H(A^{\bullet}))$. 

There is a natural isomorphism \cite[$\S 1$]{KnudsenMumford}
\begin{equation} \label{detofcohomology}
\mathrm{det}(A^{\bullet}) \otimes \mathrm{det}(H(A^{\bullet})) \cong K.
\end{equation}
We let $s_{A^{\bullet}}$ denote the preimage of $1$ under the above isomorphism.

\begin{defn}[Reidemeister torsion of a complex with volume forms]
The Reidemeister torsion $RT(A^{\bullet}, \omega_{\bullet}, \mu_{\bullet})$ is defined to be 
$$\omega \otimes \mu^{-1} (s_{A^{\bullet}}).$$
\end{defn}

It is readily checked that if $\mu'_i = c_i \mu_i$ for some non-zero constants $c_i,$ then
\begin{equation} \label{rtscaling}
RT(A^{\bullet}, \omega, \mu') = \frac{c_0 c_2 \cdots }{c_1 c_3 \cdots} \cdot RT(A^{\bullet}, \omega, \mu).
\end{equation}

\begin{exa} \label{combinatorialvolume}
{\rm Suppose $K=\C$ and $A^\bullet = C^\bullet \otimes \C$ for some finite complex of free abelian groups $C^\bullet$. A $\mathbb{Z}$-basis $(e_1, \ldots , e_n)$ of $C^i$ determines a volume form $\omega_i$ on $A^i = C^i \otimes_{\mathbb{Z}} \C$ by the formula
$$\omega_i (e_1 \wedge ... \wedge e_n) = 1.$$
The volume form $\omega_i$ is well-defined up to sign. We endow $A^{\bullet}$ with such a volume form $\omega_{\mathbb{Z}} \in \mathrm{det}(A^{\bullet})^*$. Using that $H^i(A^{\bullet}) = H^i(C^{\bullet})_\C$ we can similarly endow $H(A^{\bullet})$ with a volume form $\mu_{\mathbb{Z}} \in \mathrm{det}(H(A^{\bullet}))^*$. Then
$$|RT(A^{\bullet}, \omega_{\mathbb{Z} }, \mu_{\mathbb{Z}} )| = \frac{|H^1(C^{\bullet})_{\mathrm{tors}}| \cdot |H^3(C^{\bullet})_{\mathrm{tors}}| \cdots}{|H^0(C^{\bullet})_{\mathrm{tors}}| \cdot |H^2(C^{\bullet})_{\mathrm{tors}}| \cdots}.$$}
\end{exa}

\begin{defn}
Fix a triangulation $T$ of a Riemannian manifold $(M,g)$ together with a metrized local system of free abelian groups $L \rightarrow M.$  Let $C^i (M, L ; T)$ be the corresponding cochain (free abelian) groups. We can endow $A^\bullet := C^\bullet (M,L; T) \otimes \C$ with a combinatorial volume form $\omega_{\mathbb{Z}}$ as defined in Example \ref{combinatorialvolume}.  Identifying each $H^i(M,L_{\C})$ with the vector space of harmonic $L_{\C}$-valued $i$-forms on $M$ we define a volume form $\mu_g$ on $H^\bullet (M, L_\C) = H(A^{\bullet}).$  We finally define
$$RT(M,L) := RT(A^{\bullet} , \omega_{\mathbb{Z}}, \mu_g).$$
It follows from Example \ref{combinatorialvolume} and \eqref{rtscaling} that we have:
\begin{equation} \label{eq:RTML}
|RT(M,L)| = \frac{|H^1(M,L)_{\mathrm{tors}}| \cdot |H^3(M,L)_{\mathrm{tors}}| \cdots}{|H^0(M,L)_{\mathrm{tors}}| \cdot |H^2(M,L)_{\mathrm{tors}}| \cdots} \times \frac{R^0(M,L) R^2(M,L) \cdots}{R^1(M,L) R^3(M,L) \cdots},
\end{equation}
where $R^i(M,L)$ is the volume $\vol(H^i(M,L)_{\mathrm{free}}),$ with respect to the volume forms obtained by identifying $H^i(M, L_\C)$ with the space of harmonic $L$-valued $i$-forms on $M.$    
\end{defn}
Note that it follows in particular from \eqref{eq:RTML} that $|RT(M,L)|$ does not depend on the triangulation $T.$  The following beautiful theorem relates $|RT(M,L)|$ to the analytic torsion $T_M (L) = \exp (t_M (L))$ that we have already considered in the particular case of locally symmetric spaces\footnote{In fact we have more generally considered the {\it twisted} analytic torsion that we deal with below.}, see \cite{muller} for the general definition. 

\begin{thm}[Cheeger-M\"{u}ller theorem for unimodular local systems \cite{muller}] \label{cmtheorem}
Let $L \rightarrow M$ be a unimodular local system over a compact Riemannian manifold.  There is an equality
$$T_M (L) = |RT(M,L)|.$$
\end{thm}

\subsection{Twisted local system over locally symmetric spaces}
Let $\mathbf{G}$ be a connected semisimple quasi-split algebraic group defined over $\RR$. Let $\mathbb{E}$ be an \'{e}tale $\R$-algebra such that $\mathbb{E} / \RR$ is a cyclic Galois extension with Galois group generated by $\sigma \in \Aut(\mathbb{E} / \R).$ The automorphism $\sigma$ induces a corresponding automorphism of the group $G$ of real points of $\mathrm{Res}_{\mathbb{E}/ \RR} \mathbf{G}$. We assume that condition \eqref{condition} holds, namely: $H^1 (\sigma , G) = \{ 1\}$. 

Choose a Cartan involution $\theta$ of $G$ that commutes with $\sigma$ and denote by $K$ the group of fixed points of $\theta$ in $G$ and let $X= G /K$ be the associated symmetric space and. The involution $\sigma$ acts on $G$ and $X$; we denote by $G^{\sigma}$ and $X^{\sigma}$ the corresponding sets of fixed points. 

As above we denote by $\widetilde{G}$ the twisted space $G \rtimes \sigma$. 

Now let $\Gamma$ be a torsion free cocompact lattice of $G$ that is $\sigma$-stable and let $(\widetilde{\rho} , F)$ be a complex finite dimensional $\sigma$-stable irreducible representation of $\widetilde{G}$ defined over $\R$ such that
\begin{enumerate}
\item $\rho (\Gamma)$ stabilizes some fixed lattice $\mathcal{O}$ in the real points of $F$ \label{condition1};
\item the representation $\tilde{\rho}$ is strongly twisted acyclic (see \S \ref{def:SA}). \label{condition2} 
\end{enumerate}

The group $\Gamma$ acts on $X$ and diagonally on $X \times \mathcal{O}$ (through the representation $\rho$ on the second factor). We let
$$\mathcal{M} = \Gamma \backslash X \mbox{ and } \loc = \Gamma \backslash (X \times \mathcal{O} )$$
be the corresponding quotients. Projection on the first factor gives a unimodular local system $\loc \rightarrow \M$ of free abelian groups; it is moreover equivariant with respect to an automorphism of $\M$ of finite order.\footnote{We shall use the notation $\M$ for manifolds with an involution and the notation $M$ when there is no involution.} From now on we shall furthermore assume that the order $p$ of $\sigma$ is \emph{prime}.

We shall consider a family of covering manifolds $\mathcal{M}_n$ associated to a sequence $\{ \Gamma_n \}$ of finite index $\sigma$-stable subgroups such that for every $\delta \in \widetilde{\Gamma}$ with $\delta \notin Z^1 (\sigma , \Gamma)$ the sequence 
\begin{equation} \label{seq}
\left(\frac{| \{ \gamma \in \Gamma_n \backslash \Gamma \; : \; \gamma \delta \gamma^{-1} \in \widetilde{\Gamma}_n \}|}{[\Gamma^\sigma : \Gamma_n^\sigma]} \right)_{n \geq 0}
\end{equation}
remains bounded by some uniform constant $A,$ independent of $\delta,$ and converges to $0$ as $n$ tends to infinity. \medskip

\noindent
{\it Remark.} Many examples of sequences $\{ \Gamma_n \}$ satisfying \eqref{seq} are furnished by Proposition \ref{boundinggrowth}.  In particular, supose $G = \G(E_\R)$ for some semisimple group $\G / O_{F,S},$ where $O_{F,S}$ denotes the ring of $S$-integers in a number field $F$ and $E/F$ is a cyclic Galois extension.  Let $\Gamma \subset \G(E_\R)$ be an arbitrary $\sigma$-stable cocompact lattice.  The sequence $\{  \Gamma_\mathfrak{p} \},$ where $\Gamma_\mathfrak{p} = \{ \gamma \in \Gamma: \gamma = 1 \mod \mathfrak{p} \},$ satisfies \eqref{seq}.  See Proposition \ref{boundinggrowth} for further details and more examples of a similar flavor.

\medskip

\subsection{Equivariant Reidemeister torsion}

Let $P(x) = x^{p -1} + x^{p-2} + ... + 1.$  For a polynomial $h \in \mathbb{Z}[x]$ and a $\mathbb{Z}[\sigma]$-module $A,$ we define $A^{h(\sigma)}: = \{ a \in A: h(\sigma) \cdot a = 0 \}.$  We shall denote by $A[p^{-1}]$ the localization $S^{-1} A$ where $S$ is the multiplicative subset $\{ p^n \; : \; n \in \mathbb{Z} \} \subset \mathbb{Z} \subset \mathbb{Z} [\sigma ]$ and let $A[p^{\infty}] = \{ a \in A: p^n a = 0 \text{ for some } n \geq 0 \},$ the $p$-power torsion subgroup of $A.$ Finally if $K$ is a field we set $\loc_K := \loc \otimes_{\mathbb{Z}} K$. 

In particular $\loc_{\C}$ defines a flat complex bundle on $\M$ and the action of $\sigma$ on $\M$ lifts to the flat vector bundle $\loc_\C$. We shall apply results of Bismut and Zhang to relate equivariant combinatorial and analytical torsions of $\loc_\C \to \M$.  

First recall that it is a general result of Wasserman that there exists a $\sigma$-invariant Morse function $f : \M \to \mathbb{R}$. We shall in fact work with particular choices of Morse functions on the $\M_n$'s.  By \cite{Illman} there exists an equivariant $CW$-triangulation on $\M$. It lifts to an equivariant $CW$-triangulation on each $\M_n$. By a standard construction there correspond to these triangulations natural Morse functions $f_n : \M_n \to \R$ such that the set of critical points of $f_n$ is exactly the set of barycenters of the simplexes of the $CW$-triangulation of $\M_n$. One verifies easily that these constructions can be made $\Gamma/ \Gamma_n$-equivariantly and that the resulting functions $f_n$ can be made $\sigma$-invariant. The proof of \cite[Theorem 1.10]{BZ2} moreover implies that one may construct these $f_n$'s such that:
\begin{equation} \label{A}
\mbox{for any critical point } x \in {\rm Crit}(f_n) \cap \M_n^\sigma, \mbox{ the Hessian } d^2 f_n (x) \mbox{ is positive definite on } N_x.
\end{equation}
Here $N$ denotes the normal bundle to $\M_n^{\sigma}$ in $\M_n$. 

It then follows from \cite[Theorem 1.8]{BZ2} that one may modify the locally symmetric metric $g$ on $\M_n$ to get a metric $g'$ which equals $g$ in a neighborhood of all critical points of $f_n$ in such a way that the corresponding gradient vector field $X_n = \nabla_{g'}(f_n)$ satisfies the Smale transversality condition.

To ease notation we shall now concentrate of $\M$ and $f$ and explain when needed what happens when $(\M , f)$ is replaced by $(\M_n , f_n )$. 
Let $MS(\M ,\loc_\C )$ be the Morse-Smale complex \cite[$\S 1.6$]{BZ1} associated to the $\sigma$-invariant Morse function $f$ and the associated invariant transversal gradient vector field $X$. We endow the chain groups of $MS(\M ,\loc_\C)^{\sigma - 1}$ and $MS(\M ,\loc_\C)^{P(\sigma)}$ with volume forms induced from the metric on $\loc_\C $ and the combinatorial volume forms induced by the unstable cells of the gradient vector field \cite[$\S 1.4$]{Lip1}.  We endow the cohomology groups $H^i(\M ,\loc_{\C})^{\sigma - 1}$ and $H^i(\M ,\loc_{\C})^{P(\sigma)}$ with the metric induced by the $L^2$-metric on $\loc_\C$-valued harmonic $i$-forms on $\M$ \cite[Definition B.4]{Lip1}.

\begin{defn} \label{deftwistedrt}
The equivariant Reidemeister torsion of the equivariant local system $\loc_\C \rightarrow \M$ of $\C$-vector spaces is defined by
$$\log RT_{\sigma}(\M,\loc_\C; f,X) := \log |RT(MS(\M,\loc_\C)^{\sigma - 1})| - \frac{1}{p-1} \log |RT(MS(\M,\loc_\C)^{P(\sigma)})| .$$
\end{defn}

\medskip
\noindent
{\it Remarks.} 1. Equivariant Reidemeister torsion a priori depends on the choice of Morse theoretic data; see the remark following Theorem \ref{twistedatequalstwistedrt} for further discussion.  However, the discrepancy between equivariant Reidemeister torsion and a purely cohomological quantity can be bounded independently of the Morse theoretic data for local systems endowed with integral structure; see Theorem \ref{concretert} below.  

2. Let $G$ be the finite group generated by $\sigma$. We may define a Reidemeister torsion $\log RT_G (\M , \loc_\C ; f , X )$ with values in the complex representation ring of $G$. We could then have alternatively defined $\log RT_{\sigma}(\M,\loc_\C; f,X)$ as the trace of $\sigma$ in $\log RT_G (\M , \loc_\C ; f , X )$; this viewpoint is closer to that of \cite{BZ2}. Both definitions agree: indeed if $C$ is a finite dimensional $\mathbb{Z}[\sigma]$-module, its complexification $A=C_\C$ decomposes as a direct sum $\oplus_\chi A_{\chi}$ of isotypical subspaces indexed by characters of $\mathbb{Z}/p \mathbb{Z}$, and we have 
$$C^{\sigma} = A_1 \mbox{ and } C^{P(\sigma)} =  \oplus_{\chi \neq 1} A_{\chi}.$$
We conclude that the trace of $\sigma$ in $C$ is equal to 
$$\dim C^\sigma + (\zeta + \ldots + \zeta^{p-1}) \frac{1}{p-1} \dim C^{P(\sigma)} = \dim C^\sigma - \frac{1}{p-1} \dim C^{P(\sigma )}.$$
(Here $\zeta$ is some primive $p$th root of unity.)

\medskip

\begin{defn} \label{twistedregulator}
Let $R^i(\M,\loc)^{\sigma - 1}$ denote the covolume of the lattice $H^i(\M,\loc)^{\sigma - 1}$ in the real vector space $H^i(\M,\loc_\R)^{\sigma - 1}$ with inner product induced by that on harmonic forms.  Define $R^i(\M,\loc)^{P(\sigma)}$ similarly.  The \emph{equivariant regulator} $R^i_{\sigma}(\M,\loc)$ is defined to be 
$$R^i_{\sigma}(\M,\loc) := \frac{R^i(\M,\loc)^{\sigma - 1}}{ \left( R^i(\M,\loc)^{P(\sigma)} \right)^{\frac{1}{p-1}}}$$
\end{defn}

\begin{thm}[Concrete relationship between twisted RT and cohomology] \label{concretert}
Suppose that the fixed point set $\M^{\sigma}$ has Euler characteristic 0.  Then for an arbitrary choice of invariant Morse function $f$ and invariant transversal gradient vector field $X,$
\begin{equation*}
\begin{split}
\log RT_{\sigma}(\M,\loc_\C ; f, X) =&  - \sum_i (-1)^i \left(\log \left|H^i(\M, \loc)_{\mathrm{tors}}[p^{-1}]^{\sigma - 1} \right|  -  \frac{1}{p-1} \log \left|H^i(\M, \loc)_{\mathrm{tors}}[p^{-1}]^{P(\sigma)} \right| \right) \\
&+ \sum_i (-1)^i \log R^i_{\sigma}(\M, \loc) \\
&+ O\left( \log|H^{*}(\mathcal{M}, \loc)_{\mathrm{tors}}[p^{\infty}]| +\log |H^{*}(\mathcal{M}, \loc_{\mathbb{F}_p})| + \log |H^{*}(\M^{\sigma}, \loc_{\mathbb{F}_p})| \right)
\end{split}
\end{equation*}
\end{thm}
\begin{proof}
For $\loc_\C$ acyclic, this is proven in \cite[Corollary 3.8]{Lip1}.  As described in \cite[Proposition B.6]{Lip1}, almost exactly the same proof applies even to those $\loc$ for which $\loc_\C$ is not acyclic.    
\end{proof}

\medskip
\noindent
{\it Remark.} It follows from the proof given in \cite{Lip1} that the implicit constant in $O(\cdot )$ only depends on the dimension of $\M_n$. In particular it is independent of $n$. 

\subsection{The equivariant Cheeger-M\"{u}ller theorem}
Recall that we have defined in \eqref{TT} the twisted analytic torsion of a locally symmetric space equipped with an equivariant, metrized, unimodular local system of complex vector spaces acted on equivariantly and isometrically by $\sigma$ of finite order. Alternatively this is equal to 
$$\log T^{\sigma}_{(\M , g)} (\loc_\C ) := \frac12 \frac{\partial \theta^{\loc_\C }}{\partial s} (0) (\sigma),$$
where the function $\theta^{\loc_\C } (s) (\sigma )$ is defined in \cite[Definition 2.2]{BZ2} and $g$ denotes the locally symmetric metric on $\M$. The following theorem will be easily deduced from the deep work of Bismut and Zhang \cite{BZ2}.

\begin{thm}[Equivariant Cheeger-M\"{u}ller theorem] \label{twistedatequalstwistedrt}
Suppose that $X^{\sigma}$ is odd-dimensional. Then we have:
$$\log T^{\sigma}_{(\M , g)} (\loc_\C ) = \log RT_{\sigma}(\M,\loc_\C ; f,X).$$ 
\end{thm}
\begin{proof} First observe that it follows from \cite[Proposition 2.3]{Rohlfs} that the set of fixed points 
$$\mathcal{M}^\sigma = (\Gamma \backslash X)^{\sigma}$$
is a finite disjoint union of its connected components that are parametrized by $H^1 (\sigma , \Gamma )$. Since, under our assumption \eqref{condition}, the map 
$$H^1 (\sigma , \Gamma ) \to H^1 (\sigma , G)$$ 
has obviously trivial image, it moreover follows that each connected component is isometric to $\Gamma^{\sigma} \backslash X^\sigma$. In particular it is odd dimensional. Moreover, splitting the complex vector space $F = \oplus_j F^{\alpha_j}$ according to the eigenvalues of $\sigma$ yields a decomposition of the complex vector bundle $\loc_\C$ over $\Gamma^{\sigma} \backslash X^\sigma$ as a direct sum of complex vector bundles $\Gamma^{\sigma} \backslash (X^{\sigma} \times F^{\alpha_j} )$ that are all unimodular (note that by hypothesis $\Gamma^{\sigma}$ is torsion-free). The differential form $ \theta_{\sigma}(\loc_{\C } , h^{\loc_\C })$ defined in \cite[Definition 2.5]{BZ2} therefore vanishes.  And, by \cite[Theorem 0.2]{BZ2}, we conclude that
\begin{multline} \label{BZformula}
2 [ \log RT_{\sigma}(\M,\loc_\C ; f,X) - \log T^{\sigma}_{(\M , g)} (\loc_\C ) ] \\ 
-  \frac{1}{4} \sum_{x \in {\rm Crit}(f) \cap \M^{\sigma}} (-1)^{\ind(f|_{\M^{\sigma}},x)} \sum_j (n_{+}(\beta_j,x) - n_{-}(\beta_j,x)) \cdot C_j \cdot \tr [\sigma | \loc_x], 
\end{multline}
where:
\begin{itemize}
\item the integers $n_{+}(\beta_j,x)$ and $n_{-}(\beta_j,x)$ respectively denote the number of positive and negative eigenvalues of the Hessian $d^2 f(x)$ acting on $N(\beta_j)_x$, where $N(\beta_j )$ is the subbundle of the normal bundle $N$ to $\M^{\sigma}$ on which $\sigma$ acts by $e^{\pm i \beta_j}$; 

\item the constant $C_j$ is related to the equivariant torsion of a sphere.  See \cite[\S 11]{LottRothenberg}.
\end{itemize}
Now, for our particular choice Morse data and replacing the locally symmetric metric $g$ by the modified metric $g'$, we deduce from \eqref{A} that the expression
$$\sum_j (n_{+}(\beta_j,x) - n_{-}(\beta_j,x)) \cdot C_j \cdot \tr [\sigma | \loc_x]$$
is constant for critical points $x$ in a single connected component $\M_0$ of the fixed point set $\M^\sigma.$  Indeed, $n_{+}(\beta_j,x) = \dim N(\beta_j)|_{\M_0}$, $n_{-}(\beta_j,x) = 0,$ and, since $\sigma$ preserves the flat connection, $\tr(\sigma | \loc_x)$ is also constant.  Therefore,
\begin{multline*}
- \frac{1}{4} \sum_{x \in {\rm Crit}(f) \cap \M^{\sigma}} (-1)^{\ind(f|_{\M^{\sigma}},x)} \sum_j (n_{+}(\beta_j,x) - n_{-}(\beta_j,x)) \cdot C_j \cdot \tr [\sigma | \loc_x] \\
\begin{split}
&= \sum_{ \M_0 \in \pi_0(\M^\sigma)} \text{constant}(\M_0) \sum_{x \in {\rm Crit}(f) \cap \M_0}  (-1)^{\ind(f|_{\M_0},x)} \\
&= \sum_{\M_0 \in \pi_0(\M^\sigma)} \text{constant}(\M_0) \chi(\M_0) \\
&= 0.
\end{split}
\end{multline*}
Thus, 
$$ \log RT_{\sigma}(\M,\loc_\C ; f,X) - \log T^{\sigma}_{(\M , g ' )} (\loc_\C ) = 0.$$
To conclude, note that by the anomaly formula of Bismut-Zhang \cite[Theorem 0.1]{BZ2},\footnote{Here again we use that the differential form $ \theta_{\sigma}(\loc, h^\loc)$ vanishes.}
$$T^{\sigma}_{(\M , g ' )} (\loc_\C ) = T^{\sigma}_{(\M , g  )} (\loc_\C )$$
when all components of the fixed point set are odd-dimensional.
\end{proof}
      
\subsection{Growth of torsion in cohomology of locally symmetric spaces} \label{growthoftorsion}
Under our hypotheses it follows from Theorem \ref{twistedtorsionlimitmultiplicity} that
\begin{equation} \label{lim1}
\frac{\log T_{\Gamma_n \backslash X}^{\sigma} (\rho)}{|H^1 (\sigma , \Gamma_n ) | \vol (\Gamma_n^{\sigma} \backslash G^{\sigma})} \to t_X^{(2) \sigma} (\rho) .
\end{equation}

From now on we will furthermore assume that $\delta (G^\sigma) = 1$; note that this forces $X^{\sigma}$ to be odd dimensional. Theorem \ref{twistedatequalstwistedrt} therefore implies that 
$$\log T_{\Gamma_n \backslash X}^{\sigma} (\rho) = \log RT_{\sigma}(\M_n ,\loc_{\C } ; f_n ,X_n ).$$ 
Then Theorem \ref{concretert}, the remark following it, and Equation \eqref{lim1} imply the following

\begin{cor} \label{refinedcohomologygrowth}
Under the above hypotheses (in particular with $\sigma$ of \emph{prime order} $p$) suppose furthermore that
$$\log |H^{*}(\Gamma_n, \mathcal{O})[p^{\infty}]| = o(|H^1 (\sigma , \Gamma_n ) |\vol(\Gamma_n^{\sigma} \backslash G^\sigma )) \mbox{ and } \log |H^* (\Gamma_n^{\sigma}, \mathcal{O}_{\mathbb{F}_p})| = o(\vol(\Gamma_n^{\sigma} \backslash G^\sigma )).$$
Then 
$$\frac{\sum_i (-1)^i \log R^i_{\sigma}(\Gamma_n ,\rho ) - \sum_i (-1)^i \left(\log \left|H^i(\Gamma_n, \mathcal{O})[p^{-1}]^{\sigma - 1} \right|  -  \frac{1}{p-1} \log \left|H^i(\Gamma_n , \mathcal{O} )[p^{-1}]^{P(\sigma)} \right| \right)}{|H^1 (\sigma , \Gamma_n ) | \vol (\Gamma_n^{\sigma} \backslash G^{\sigma})} \xrightarrow{n \rightarrow \infty}  t_X^{(2) \sigma} (\rho).$$
\end{cor}

In particular, if $\rho$ is strongly acyclic, then 
$$\frac{ - \sum_i (-1)^i \left(\log \left|H^i(\M, \loc)[p^{-1}]^{\sigma - 1} \right|  -  \frac{1}{p-1} \log \left|H^i(\M, \loc)[p^{-1}]^{P(\sigma)} \right| \right)}{|H^1 (\sigma , \Gamma_n ) | \vol (\Gamma_n^{\sigma} \backslash G^{\sigma})} \xrightarrow{n \rightarrow \infty}  t_X^{(2) \sigma} (\rho).$$

One can deduce from this an unconditional cohomology growth result:

\begin{cor} \label{unrefinedcohomologygrowth}
Enforce the same notations and hypotheses as in Corollary \ref{refinedcohomologygrowth}, with \emph{no a priori cohomology growth assumptions}.
Furthermore, assume that $\rho$ is strongly acyclic and that $\mathcal{O}_{\mathbb{F}_p}$ is trivial. Then
$$\limsup \frac{\sum_i \log | H^{i}(\Gamma_n, \mathcal{O})_{\mathrm{tors}}|}{\vol (\Gamma_n^{\sigma} \backslash G^{\sigma})} > 0.$$
\end{cor}

\medskip
\noindent
{\it Remark.}  The local system $(\Gamma_n \backslash X, \mathcal{O}_{\mathbb{F}_p})$ is trivial if and only if $\Gamma_n$ is contained in the kernel of $\rho$ mod $p.$  We can therefore construct many examples using Proposition \ref{boundinggrowth} relative to the over group $\Gamma = \ker(\rho \mod p).$

\begin{proof} 
Suppose that not both of the growth hypotheses of Corollary \ref{refinedcohomologygrowth} hold.
\begin{itemize}
\item
If $\limsup \frac{\log |H^{*}(\Gamma_n, \mathcal{O})[p^{\infty}]|}{\vol (\Gamma_n^{\sigma} \backslash G^{\sigma})} > 0,$ we are done.
\item
Suppose the mod $p$ cohomology of $((\Gamma_n \backslash X)^{\sigma}, \mathcal{O}_{\mathbb{F}_p})$ is large, i.e. 
$$\frac{\log |H^{*}((\Gamma_n \backslash X)^\sigma, \mathcal{O}_{\mathbb{F}_p}) |}{ |H^1(\sigma,\Gamma_n)| \vol (\Gamma_n^{\sigma} \backslash G^{\sigma}) } \nrightarrow 0.$$
Because $\mathcal{O}_{\mathbb{F}_p}|_{(\Gamma_n \backslash X)^\sigma }$ is trivial,  the latter implies that
$$\frac{\log |H^{*}((\Gamma_n \backslash X)^\sigma, \mathbb{F}_p) |}{ |H^1(\sigma,\Gamma_n)| \vol (\Gamma_n^{\sigma} \backslash G^{\sigma}) } \nrightarrow 0.$$
The conclusion follows by Smith theory \cite[$\S$ III]{bredon}.  Indeed \cite[$\S$ III.4.1]{bredon}, 
\begin{equation} \label{smith}
\dim_{\mathbb{F}_p} H^{*}( (\Gamma_n \backslash X)^\sigma, \mathbb{F}_p) \leq \dim_{\mathbb{F}_p} H^{*}(\Gamma_n \backslash X , \mathbb{F}_p) = \frac{1}{\mathrm{rank} \ \mathcal{O}} \dim_{\mathbb{F}_p} H^{*}(\Gamma_n \backslash X ,  \mathcal{O}_{\mathbb{F}_p} ), 
\end{equation}
where the final equality follows because $\mathcal{O}_{\mathbb{F}_p}$ is trivial.  Because $\rho$ is strongly acyclic, $(\Gamma_n \backslash X ,\mathcal{O})$ has no rational cohomology.  The desired conclusion then follows by the universal coefficient theorem. 
\end{itemize}
  
Otherwise, both a priori cohomology growth hypothesis from Corollary \ref{refinedcohomologygrowth} are satisfied. We thus apply Corollary \ref{refinedcohomologygrowth}, whose conclusion is more refined. Note that it follows from Theorems \ref{twistedtorsionproduct} and \ref{torsioncomparison} (and the computations in the non-twisted case done in \cite{BV}) that $t^{(2) \sigma}_{X}(\rho) \neq 0$ whenever $\delta (G^{\sigma}) =1$. 
\end{proof}

\medskip
\noindent
{\it Remark.} Corollaries \ref{refinedcohomologygrowth} and \ref{unrefinedcohomologygrowth} were one major source of inspiration for this paper.  We sought to understand when $t^{(2) \sigma}_{X}(\rho) \neq 0$ in order to detect torsion cohomology growth. Note that in Corollary \ref{unrefinedcohomologygrowth} we have assumed that the sequences \eqref{seq} are uniformly bounded and tends to $0$ with $n$. Proposition \ref{boundinggrowth} verifies this assumption when the $\Gamma_n$ are principal congruence subgroups of prime levels. Theorem \ref{T:14} of the Introduction therefore follows from Corollary \ref{unrefinedcohomologygrowth}.

\subsection{Comparison to $p$-adic methods}
Let $\G$ be an algebraic group over $\mathbb{Z}$ which is smooth over $\mathbb{Z}[N^{-1}]$ and for which $\G_\mathbb{Q}$ is semisimple.  As a byproduct of their study of completed cohomology \cite{CEcompletedcohomology}, Calegari and Emerton are able to prove non-trivial upper and lower bounds on cohomology growth for the family of groups $\Gamma_{p^n} = \ker \left(\G(\mathbb{Z}) \rightarrow \G(\mathbb{Z} / p^n \mathbb{Z}) \right)$ for any prime $p \nmid N.$  Using Poincar\'{e} duality for completed cohomology, they show 
$$\dim_{\mathbb{F}_p} H^{*}(\Gamma_{p^n},\mathbb{F}_p) \gg [\Gamma_1:\Gamma_{p^n}]^{1 - \alpha}$$
where $\alpha = \frac{\dim (G/K)}{\dim G}$; Calegari and Emerton prove this for $\{ \Gamma_{p^n} \}$ a family of 3-manifold groups in \cite{CEhyperbolic3manifold} and Calegari extends this to general $\G$ in \cite{CalegariBlog}.  For any local system $\loc$ arising from a representation of $\G$ defined over $\mathbb{Q}$ with $\loc_{\mathbb{Q}}$ acyclic, they deduce that
\begin{equation} \label{completedcohomologylowerbound}
\log |H^{*}(\Gamma_{p^n},\loc)_{\mathrm{tors}}| \geq \log |H^{*}(\Gamma_{p^n}, \loc)[p^{\infty}]| \gg [\Gamma_1:\Gamma_{p^n}]^{1 - \alpha}
\end{equation}
as an immediate consequence.  It is noteworthy that $\alpha = \frac{1}{2}$ if $G$ is a complex Lie group and $1 - \alpha \geq \frac{1}{3}$ for arbitrary $G.$  The resulting lower bound obtained by \eqref{completedcohomologylowerbound} is of the same quality as that proven in Corollary \eqref{unrefinedcohomologygrowth} for quadratic base change of groups with $\delta(G^{\sigma}) = 1$ and is always strictly larger for cyclic base change of degree greater than two for groups with $\delta(G^{\sigma}) = 1.$  Nonetheless, these lower bounds do not subsume our main theorems on torsion cohomology growth.

\begin{itemize}
\item
Suppose $\delta(G^{\sigma}) = 1.$  By Corollary \ref{unrefinedcohomologygrowth}, any family of groups satisfying the hypotheses of Theorem \ref{twistedtorsionlimitmultiplicity} exhibits torsion cohomology growth.  As shown in Proposition \ref{boundinggrowth}, fairly general families $\{ \Gamma'_q \}$ of congruence subgroups which grow horizontally, e.g. as $q$ varies through a sequence of primes, satisfy these hypotheses.  However, \eqref{completedcohomologylowerbound} does not give any information concerning cohomology growth for such families of horizontally growing congruence subgroups.   

\item
Suppose $\delta(G^{\sigma}) = 1$ and consider the family of congruence subgroups $\{ \Gamma_{p^n} \}$ of full level $p^n.$  Both Corollary \ref{refinedcohomologygrowth} and \eqref{completedcohomologylowerbound} yield lower bounds on torsion cohomology growth.  However, their origins should be regarded as very distinct.

Cohomology classes accounted for by \eqref{completedcohomologylowerbound} are an aggregate of mod $p$ congruences between (mod $p$) automorphic representations of $\G$ of arbitrary level.  

On the other hand, the cohomology classes accounted for by Corollary \ref{refinedcohomologygrowth} conjecturally arise by base change transfer over $\mathbb{Z}$ \cite{CalegariVenkatesh} \cite{Lip2}.  Partial evidence for this transfer occurs in Theorem \ref{torsioncomparison}, which may be regarded as ``numerical base change transfer over $\mathbb{Z}$ at infinite level" (see \cite{Lip2} for some special cases of numerical base change transfer over $\mathbb{Z}$ at finite level). 

Base change for torsion cohomology leads us to expect that torsion witnessed in Corollary \ref{refinedcohomologygrowth} is supported at the same primes as torsion in the cohomology of locally symmetric spaces for $G^{\sigma}$; computations suggest that the latter primes are large and irregular \cite{SengunBianchi}.  On the other hand, torsion witnessed through \eqref{completedcohomologylowerbound} is supported at a single prime $p$ and gives no information about the prime to $p$ part of torsion cohomology.
\end{itemize}

\medskip

\bibliography{bibli}

\def\cftil#1{\ifmmode\setbox7\hbox{$\accent"5E#1$}\else
  \setbox7\hbox{\accent"5E#1}\penalty 10000\relax\fi\raise 1\ht7
  \hbox{\lower1.15ex\hbox to 1\wd7{\hss\accent"7E\hss}}\penalty 10000
  \hskip-1\wd7\penalty 10000\box7}
\begin{thebibliography}{10}

\bibitem{7samurai}
M.~{Abert}, N.~{Bergeron}, I.~{Biringer}, T.~{Gelander}, N.~{Nikolov},
  J.~{Raimbault}, and I.~{Samet}.
\newblock {On the growth of $L^2$-invariants for sequences of lattices in {L}ie
  groups}.
\newblock {\em ArXiv e-prints}, October 2012.

\bibitem{AC}
James Arthur and Laurent Clozel.
\newblock {\em Simple algebras, base change, and the advanced theory of the
  trace formula}, volume 120 of {\em Annals of Mathematics Studies}.
\newblock Princeton University Press, Princeton, NJ, 1989.

\bibitem{BM}
Dan Barbasch and Henri Moscovici.
\newblock {$L^{2}$}-index and the {S}elberg trace formula.
\newblock {\em J. Funct. Anal.}, 53(2):151--201, 1983.

\bibitem{BV}
Nicolas Bergeron and Akshay Venkatesh.
\newblock The asymptotic growth of torsion homology for arithmetic groups.
\newblock {\em J. Inst. Math. Jussieu}, 12(2):391--447, 2013.

\bibitem{BZ2}
J.-M. Bismut and W.~Zhang.
\newblock Milnor and {R}ay-{S}inger metrics on the equivariant determinant of a
  flat vector bundle.
\newblock {\em Geom. Funct. Anal.}, 4(2):136--212, 1994.

\bibitem{BZ1}
Jean-Michel Bismut and Weiping Zhang.
\newblock An extension of a theorem by {C}heeger and {M}\"uller.
\newblock {\em Ast\'erisque}, (205):235, 1992.
\newblock With an appendix by Fran{\c{c}}ois Laudenbach.

\bibitem{BLS}
A.~Borel, J.-P. Labesse, and J.~Schwermer.
\newblock On the cuspidal cohomology of {$S$}-arithmetic subgroups of reductive
  groups over number fields.
\newblock {\em Compositio Math.}, 102(1):1--40, 1996.

\bibitem{Bouaziz}
Abderrazak Bouaziz.
\newblock Formule d'inversion d'int\'egrales orbitales tordues.
\newblock {\em Compositio Math.}, 81(3):261--290, 1992.

\bibitem{bredon}
Glen~E. Bredon.
\newblock {\em Introduction to compact transformation groups}.
\newblock Academic Press, New York-London, 1972.
\newblock Pure and Applied Mathematics, Vol. 46.

\bibitem{CalegariBlog}
F.~{Calegari}.
\newblock {Blog post: torsion in the cohomology of co-compact arithmetic
  lattices.}
\newblock {\em
  http://galoisrepresentations.wordpress.com/2013/02/06/torsion-in-the-cohomology-of-co-compact-arithmetic-lattices/}.

\bibitem{CalegariVenkatesh}
F.~{Calegari} and A.~{Venkatesh}.
\newblock {A torsion Jacquet--Langlands correspondence}.
\newblock {\em ArXiv e-prints}, December 2012.

\bibitem{CEcompletedcohomology}
Frank Calegari and Matthew Emerton.
\newblock Bounds for multiplicities of unitary representations of cohomological
  type in spaces of cusp forms.
\newblock {\em Ann. of Math. (2)}, 170(3):1437--1446, 2009.

\bibitem{CEhyperbolic3manifold}
Frank Calegari and Matthew Emerton.
\newblock Mod-{$p$} cohomology growth in {$p$}-adic analytic towers of
  3-manifolds.
\newblock {\em Groups Geom. Dyn.}, 5(2):355--366, 2011.

\bibitem{Clozel}
Laurent Clozel.
\newblock Changement de base pour les repr\'esentations temp\'er\'ees des
  groupes r\'eductifs r\'eels.
\newblock {\em Ann. Sci. \'Ecole Norm. Sup. (4)}, 15(1):45--115, 1982.

\bibitem{Delorme}
P.~Delorme.
\newblock Th\'eor\`eme de {P}aley-{W}iener invariant tordu pour le changement
  de base {${\bf C}/{\bf R}$}.
\newblock {\em Compositio Math.}, 80(2):197--228, 1991.

\bibitem{EGA}
A.~Grothendieck.
\newblock \'{E}l\'ements de g\'eom\'etrie alg\'ebrique. {I}. {L}e langage des
  sch\'emas.
\newblock {\em Inst. Hautes \'Etudes Sci. Publ. Math.}, (4):228, 1960.

\bibitem{HC}
Harish-Chandra.
\newblock Harmonic analysis on real reductive groups. {III}. {T}he
  {M}aass-{S}elberg relations and the {P}lancherel formula.
\newblock {\em Ann. of Math. (2)}, 104(1):117--201, 1976.

\bibitem{HerbWolf}
Rebecca~A. Herb and Joseph~A. Wolf.
\newblock The {P}lancherel theorem for general semisimple groups.
\newblock {\em Compositio Math.}, 57(3):271--355, 1986.

\bibitem{Illman}
S{\"o}ren Illman.
\newblock Smooth equivariant triangulations of {$G$}-manifolds for {$G$} a
  finite group.
\newblock {\em Math. Ann.}, 233(3):199--220, 1978.

\bibitem{KnudsenMumford}
Finn~Faye Knudsen and David Mumford.
\newblock The projectivity of the moduli space of stable curves. {I}.
  {P}reliminaries on ``det'' and ``{D}iv''.
\newblock {\em Math. Scand.}, 39(1):19--55, 1976.

\bibitem{Labesse}
J.-P. Labesse.
\newblock Pseudo-coefficients tr\`es cuspidaux et {$K$}-th\'eorie.
\newblock {\em Math. Ann.}, 291(4):607--616, 1991.

\bibitem{LabesseWaldspurger}
Jean-Pierre Labesse and Jean-Loup Waldspurger.
\newblock {\em La formule des traces tordue d'apr\`es le {F}riday {M}orning
  {S}eminar}, volume~31 of {\em CRM Monograph Series}.
\newblock American Mathematical Society, Providence, RI, 2013.
\newblock With a foreword by Robert Langlands [dual English/French text].

\bibitem{Lan}
Robert~P. Langlands.
\newblock {\em Base change for {${\rm GL}(2)$}}, volume~96 of {\em Annals of
  Mathematics Studies}.
\newblock Princeton University Press, Princeton, N.J.; University of Tokyo
  Press, Tokyo, 1980.

\bibitem{Lip2}
Michael Lipnowski.
\newblock Equivariant torsion and base change.
\newblock {\em Algebra Number Theory}, 9(10):2197--2240, 2015.

\bibitem{Lip1}
Michael Lipnowski.
\newblock The equivariant {C}heeger--{M}\"uller theorem on locally symmetric
  spaces.
\newblock {\em J. Inst. Math. Jussieu}, 15(1):165--202, 2016.

\bibitem{LottRothenberg}
John Lott and Mel Rothenberg.
\newblock Analytic torsion for group actions.
\newblock {\em J. Differential Geom.}, 34(2):431--481, 1991.

\bibitem{luck}
Wolfgang L{\"u}ck.
\newblock Analytic and topological torsion for manifolds with boundary and
  symmetry.
\newblock {\em J. Differential Geom.}, 37(2):263--322, 1993.

\bibitem{muller}
Werner M{\"u}ller.
\newblock Analytic torsion and {$R$}-torsion for unimodular representations.
\newblock {\em J. Amer. Math. Soc.}, 6(3):721--753, 1993.

\bibitem{MuellerPfaff}
Werner M{\"u}ller and Jonathan Pfaff.
\newblock On the asymptotics of the {R}ay-{S}inger analytic torsion for compact
  hyperbolic manifolds.
\newblock {\em Int. Math. Res. Not. IMRN}, (13):2945--2983, 2013.

\bibitem{Olbricht}
Martin Olbrich.
\newblock {$L^2$}-invariants of locally symmetric spaces.
\newblock {\em Doc. Math.}, 7:219--237, 2002.

\bibitem{RohlfsSpeh}
J.~Rohlfs and B.~Speh.
\newblock Lefschetz numbers and twisted stabilized orbital integrals.
\newblock {\em Math. Ann.}, 296(2):191--214, 1993.

\bibitem{Rohlfs}
J{\"u}rgen Rohlfs.
\newblock Lefschetz numbers for arithmetic groups.
\newblock In {\em Cohomology of arithmetic groups and automorphic forms
  ({L}uminy-{M}arseille, 1989)}, volume 1447 of {\em Lecture Notes in Math.},
  pages 303--313. Springer, Berlin, 1990.

\bibitem{SengunBianchi}
Mehmet~Haluk {\c{S}}eng{\"u}n.
\newblock On the integral cohomology of {B}ianchi groups.
\newblock {\em Exp. Math.}, 20(4):487--505, 2011.

\bibitem{Serre}
Jean-Pierre Serre.
\newblock {\em Cohomologie galoisienne}, volume~5 of {\em Lecture Notes in
  Mathematics}.
\newblock Springer-Verlag, Berlin, fifth edition, 1994.

\bibitem{Shintani}
Takuro Shintani.
\newblock Two remarks on irreducible characters of finite general linear
  groups.
\newblock {\em J. Math. Soc. Japan}, 28(2):396--414, 1976.

\end{thebibliography}

\bibliographystyle{plain}

\end{document}